\documentclass[reqno]{amsart}

\numberwithin{equation}{section}

\usepackage[latin1]{inputenc}
\usepackage[english]{babel}

\usepackage{amsmath,amsthm,amsfonts,latexsym,amssymb}
\usepackage[colorlinks]{hyperref}
\hypersetup{linkcolor=blue,citecolor=blue,filecolor=black,urlcolor=blue}
\usepackage{comment}

{ \theoremstyle{plain}
\newtheorem{theorem}{Theorem}[section]

\newtheorem{lemma}[theorem]{Lemma}
\newtheorem{corollary}[theorem]{Corollary}
  \theoremstyle{remark}
\newtheorem{remark}[theorem]{Remark}
  \theoremstyle{definition}

}

\def\R{\mathbb{R}}

\def\N{\mathbb{N}}

\DeclareMathOperator{\Ker}{Ker}
\DeclareMathOperator{\Rk}{Range}
\DeclareMathOperator{\spann}{span}

\def\eps{\varepsilon}

\def\a{\alpha}
\def\b{\beta}

\def\l{\lambda}

\def\eps{\varepsilon}

\def\Bcal{\mathcal{B}}
\def\Ccal{\mathcal{C}}
\def\Ncal{\mathcal{N}}

\begin{document}

\title[]{Multiple positive solutions of the stationary Keller-Segel system}

\thanks{D.B. is supported by INRIA - Team MEPHYSTO, MIS F.4508.14 (FNRS), PDR T.1110.14F (FNRS) 
\& ARC AUWB-2012-12/17-ULB1- IAPAS. B. Noris is partially supported by the project ERC Advanced Grant  2013 n. 339958: ``Complex Patterns for Strongly Interacting Dynamical Systems - COMPAT''}

\author[D. Bonheure, J.B. Casteras, B. Noris]{Denis Bonheure, Jean-Baptiste Casteras and Benedetta Noris}
\address{D\'epartement de Math\'ematique, Universit\'e libre de Bruxelles, Campus de la Plaine CP
213, Bd. du Triomphe, 1050 Bruxelles, Belgium}
\email{dbonheure@ulb.ac.be}
%
\email{jeanbaptiste.casteras@gmail.com}
%
\email{benedettanoris@gmail.com\\ }

\date{\today }

\begin{abstract}
We consider the stationary Keller-Segel equation 
\begin{equation*}
\begin{cases}
-\Delta v+v=\lambda e^v, \quad v>0 \quad & \text{in }\Omega,\\ 
\partial_\nu v=0 &\text{on } \partial \Omega,
\end{cases}
\end{equation*}
where $\Omega$ is a ball. In the regime $\lambda\to 0$, we study the radial bifurcations and we construct radial solutions by a gluing variational method. For any given $n\in\N_0$, we build a solution having multiple layers at $r_1,\ldots,r_n$ by which we mean that the solutions concentrate on the spheres of radii $r_i$ as $\lambda\to 0$ (for all $i=1,\ldots,n$). A remarkable fact is that, in opposition to previous known results, the layers of the solutions do not accumulate to the boundary of $\Omega$ as $\lambda\to 0$. Instead they satisfy an optimal partition problem in the limit.  
\end{abstract}

\maketitle

\section{Introduction}
One of the simplest mechanisms for aggregation of biological species is chemotaxis. This term refers to a situation where organisms move toward high concentrations of the chemical which they secrete. Keller and Segel \cite{Keller} introduced a basic model in chemotaxis. It is an advection-diffusion system consisting of two coupled parabolic equations which reads as 
\begin{equation}\label{kssys}
\begin{cases}
\dfrac{\partial u}{\partial t}=\Delta u - \nabla \cdot (u\nabla \phi(v)) \quad &\text{in } \Omega ,\\  
\dfrac{\partial v}{\partial t}=D_1 \Delta v-D_2 v+D_3 u \quad &\text{in } \Omega ,\\ 
\partial_\nu u=\partial_\nu v=0 \quad &\text{on } \partial\Omega, \\
u,v>0 &\text{in } \Omega ,
\end{cases}
\end{equation}
where $\Omega \subset \R^N$ is a smooth bounded domain, $D_i$, $i=1,\ldots ,3$ are positive constants and $\phi$ is a smooth strictly increasing function. In the previous system, $u(x,t)$ represents the concentration of the considered organisms and $v(x,t)$ the one of the chemical released. A very important property of this system is the so-called chemotactic collapse. This term refers to the fact that the whole population of organisms concentrate at  a single point in finite or infinite time. The case $\phi (v)=v$ has been studied in great details in the litterature. It is well-known that the chemotactic collapse depends strongly on the dimension of the space. When $N=1$ and $D_i \neq 0$, finite-time blow-up never occurs, whereas it always occurs when $N\geq 3$. The case $N=2$ is critical: if the initial distribution of organisms exceeds a certain threshold, then the solutions may blow-up in finite-time, whereas, if the initial mass is below this threshold, the solutions exists globally (see \cite{MR1610709}, \cite{MR1657160}).
We refer to the two surveys \cite{MR2013508}, \cite{MR2073515} and the references therein for more details.

Steady states of \eqref{kssys} are of basic importance for the understanding of the global dynamics of the system. An important remark is that the static system can be reduced to a scalar equation depending on the function $\phi$. It is easy to check that the steady states satisfy the relation :
$$\nabla\cdot (u \nabla (\log u - \phi (v))=0, $$
which, together with the boundary conditions, implies that $u=C e^{\phi (v)}$ for some positive constant $C$. Therefore, if $\phi (v)=v$, we see that \eqref{kssys} is equivalent to the so-called Keller-Segel equation
\begin{equation}\label{ks}
\begin{cases}
-\varepsilon^2 \Delta v+v=\lambda e^v, \quad v>0 \quad & \text{in }\Omega,\\ 
\partial_\nu v=0 &\text{on } \partial \Omega,
\end{cases}
\end{equation}
with $u=Ce^v$, whereas, if $\phi (v)=\ln v$, one recovers the Lin-Ni-Takagi equation
\begin{equation}\label{lnt}
\begin{cases}
-\tilde{\varepsilon}^2 \Delta v+v=v^p, \quad v>0 \quad & \text{in }\Omega,\\ 
\partial_\nu v=0 &\text{on } \partial \Omega,
\end{cases}
\end{equation}
with $u=Cv$.
In the two previous equations, the constants $\varepsilon,\tilde{\varepsilon},\lambda$ and $p$ are depending on the parameters $D_i$ of the system. A large amount of literature has been devoted to the Lin-Ni-Takagi equation in the case when $N\geq 3$ and $1<p\leq \dfrac{N+2}{N-2}$. This equation has been first consider in \cite{MR929196}. For $1<p<\dfrac{N+2}{N-2}$, Lin, Ni and Takagi showed that a mountain pass, or least energy solution $v_{\tilde{\varepsilon}}$ of equation \eqref{lnt} for $\tilde{\varepsilon}\rightarrow 0$ must behave like
$$v_{\tilde{\varepsilon}}(x)\approx \tilde{\varepsilon}^{\frac{1}{p-1}} V(\sqrt{\tilde{\varepsilon}} (x-x_{\tilde{\varepsilon}})),$$
where $V$ is the unique radial solution of 
\begin{equation}
\label{groundstate}
-\Delta V+V=V^p\ in \ \R^N,\ \lim_{|x|\rightarrow \infty}V(x)=0,
\end{equation}
and where $x_{\tilde{\varepsilon}}$ converges to a point of $\partial \Omega$ which maximizes the mean curvature. Solutions concentrating to one or several interior or boundary points have been obtained in \cite{MR1721719}, \cite{MR1696122}. When $p$ is critical, namely $p=\dfrac{N+2}{N-2}$, the situation is quite different. In this case, concentrating solutions have the following asymptotical behavior when $\tilde{\varepsilon} \rightarrow 0$,
$$v_{\tilde{\varepsilon}}(x)\approx \tilde{\varepsilon}^{-\frac{N-2}{4}} \tilde{V}(\sqrt{\tilde{\varepsilon}} (x-x_{\tilde{\varepsilon}})),$$
where $\tilde{V}$ is the standard bubble, namely the unique (up to scalings and translations) solution of
$$-\Delta \tilde{V}=\tilde{V}^{\frac{N+2}{N-2}}\ in\ \R^N.$$
In this situation, the existence of concentrating solutions depends strongly on the dimension. We refer to the very recent paper \cite{del2015interior} for more details. In all the results described before, the concentration set is zero dimensional. The question of constructing solutions concentrating on higher dimensional sets has been investigated in this last decade. In this direction, Malchiodi, Ni and Wei \cite{MalchiodiNiWei2005} obtained the existence of radial solutions concentrating on an arbitrary number of spheres $\bigcup_{j=1}^k \{|x|=r_j^{\tilde{\varepsilon}}\}$, with $1>r_1^{\tilde{\varepsilon}} >\ldots > r_k^{\tilde{\varepsilon}}>0$, in the case $\Omega = B_1(0)\subset \R^N$ and $p>1$. These solutions are called multi-layers.
An interesting feature of their result is that the radii where the concentration occurs accumulate to the boundary of the domain 
as $\tilde{\varepsilon}\rightarrow 0$. We refer to \cite{MalchiodiMontenegro2002}, \cite{MR2296306}, \cite{MR3274758} for more general constructions (considering non radial domains and more general concentration sets).

  Very recently, the Lin-Ni-Takagi equation in the case $\tilde{\varepsilon}=1$ and $p\rightarrow \infty$ has been investigated. A bifurcation analysis with respect to the parameter $p$ has been done in \cite{BonheureGrumiauTroestler2015}. We want to mention that, when $\Omega =B_1 (0)\subset \R^N$, radial solutions concentrating on spheres have been constructed in \cite{BonheureGrossiNorisTerracini2015}. We will described these solutions more carefully later, but let us notice that, in contrast to the result of Malchiodi, Ni and Wei \cite{MalchiodiNiWei2005}, the spheres where the concentration takes place do not accumulate to the boundary, but either converge to a limit configuration which satisfies an optimal partition problem.
	
Relatively less is known for the Keller-Segel equation \eqref{ks}. To the authors' knowledge, the case where $\varepsilon \rightarrow 0$ and $\lambda=1$ has been only consider in \cite{MR1644794}, where the authors obtained the same kind of results as the ones of Ni and Takagi \cite{MR929196} i.e. they prove that a mountain pass, or least energy solution of \eqref{ks} has to concentrate at a point on the boundary of $\Omega$.
	
In the case when $\varepsilon=1$ and $N=2$, the first existence result has been obtained by Wang and Wei \cite{WangWei2002}, and independently by Senba and Suzuki \cite{MR1769174}. The authors proved that given any positive number $m \in (0, (1+\lambda_1)|\Omega|)\setminus \{4k\pi\}_{k\in \N}$, where $\lambda_1 $ is the first positive eigenvalue of $-\Delta$ with Neumann boundary condition, there exists a non-constant solution of \eqref{ks} (with $\varepsilon=1$ and $N=2$) whose mass satisfies $\int_\Omega v\,dx=m$. 
The asymptotic behavior of the solutions with finite mass as $\lambda \rightarrow 0$ has been characterized by Senba and Suzuki in \cite{MR1769174}. The authors proved that if $v_\lambda$ is a family of solutions of \eqref{ks} (with $\varepsilon=1$ and $N=2$) such that
$$\lim_{\lambda \rightarrow 0} \lambda \int_\Omega e^{v_\lambda}=C_0>0,$$
then $C_0=4\pi (2k+l)$ for some positive integers $k$ and $l$. More precisely, they showed that there exist points $\xi_i \in \Omega$, $i\leq k$ and $\xi_i \in \partial \Omega$, $k<i\leq n$ for which
 $$
 v_\lambda (x)\rightarrow \sum_{i=1}^k 8\pi \mathcal{G} (x,\xi_i)+\sum_{j=k+1}^n 4\pi \mathcal{G}(x,\xi_i),
 $$
uniformly on the compact subsets of $\overline{\Omega}\setminus \{\xi_1,\ldots ,\xi_n\}$, as $\lambda\to0$. Here for $y\in \overline{\Omega}$, $\mathcal{G}(x,y)$ is the Green's function of the problem
$$
-\Delta_x \mathcal{G}+\mathcal{G}=\delta_y \text{ in } \Omega,\quad 
\dfrac{\partial \mathcal{G}}{\partial \nu_x}=0 \text{ on }\partial \Omega .
$$
Moreover, they showed that the n-tuple $(\xi_1 ,\ldots , \xi_n)$ is a critical point of a certain functional depending on the previous Green's function.
The counterpart of this result has been obtained by Del Pino and Wei in \cite{MR2209293}, where the authors construct solutions of \eqref{ks} (with $\varepsilon=1$ and $N=2$) with masses arbitrarily close to $4k\pi$, $k\in\N$. 

In this situation, it is also natural to investigate the existence of solutions concentrating on higher dimensional sets (corresponding to solution with infinite mass). In this direction, Pistoia and Vaira in \cite{PistoiaVaira2015} constructed a family $v_\lambda$ of radial solutions of
\begin{equation}\label{eq:intro_v_lambda}
\begin{cases}
-\Delta v+v=\l e^v, \quad v>0 &\text{in } B_1 \\
\partial_\nu v=0 \quad &\text{on } \partial B_1,
\end{cases}
\end{equation}
blowing-up on all the boundary of $B_1$. Here $B_1=B_1(0)$ is the unitary ball in $\R^N$, $N\geq 2$. 
Their solutions have the following asymptotic behavior
$$\lim_{\lambda \rightarrow 0} \varepsilon_\lambda v_\lambda=\sqrt{2} W,$$
where $\varepsilon_\lambda \approx -\frac{1}{\ln \lambda} $ and $W$ is the unique solution of 
$$-\Delta W+W=0\text{ in } B_1,\quad W=1\text{ on }\partial B_1 ,$$
$C^0$ uniformly on the compact subsets of $B_1$, whereas on $\partial B_1$, up to a  rescaling, they look like the one-dimensional standard bubble 
$$-V^{\prime \prime}=e^V \text{ in }\R , \quad \int_{\R} e^V\,dx<\infty.$$
The construction makes use of the Lyapunov-Schmidt reduction method.
Recently, a similar result has been obtained for a general smooth 2-dimensional domain $\Omega$ in \cite{del2014large}.
 
The aim of the present paper is to investigate the structure of the set of solutions of \eqref{eq:intro_v_lambda}, and to detect, in some cases, properties like a priori bounds, nondegeneracy and asymptotic behaviour as $\lambda\to0$. We discover that, even in the radial case, the set of solutions of \eqref{eq:intro_v_lambda} exhibits a very rich structure. This naturally leads to many new questions and open problems arise.

Let $\l_i^{rad}$ be the $i$-th eigenvalue of the operator $-\Delta+Id$ in $B_1$ with Neumann boundary conditions, restricted to the radial functions. We prove that for every $\lambda<\l_i^{rad}e^{-\l_i^{rad}}$ there exist at least $i-1$ nonconstant solutions of \eqref{eq:intro_v_lambda} ($i\geq2$). The solutions that we find can be divided into two groups: those having a local minimum at the origin, which enjoy some uniform a priori bounds in the form of Lemma \ref{lemma:a_priori_bounds} below, and those having a local maximum at the origin, which present a more singular behaviour. The solution found by Pistoia and Vaira in \cite{PistoiaVaira2015} belongs to the first group. We show that, for every $\lambda<\l_i^{rad}e^{-\l_i^{rad}}$, there exist at least $i$ different solutions of \eqref{eq:intro_v_lambda} belonging to the first group. While the solution in \cite{PistoiaVaira2015} is monotone decreasing, our solutions present an oscillatory behaviour for $i\geq 3$. If we consider the results \cite{MR2209293} 
restricted to the radial case $\Omega=B_1$, the solution found therein has a local maximum at the origin, so that it belongs to the second group. Also for the second group, we prove the existence of many solutions presenting an oscillatory behaviour. Unlike the last mentioned papers, our result holds in any dimension.

Our aim is twofold. In the first part of the paper we perform a bifurcation analysis for the problem \eqref{eq:intro_v_lambda} with respect to the parameter $\lambda$, thus detecting the two groups of solutions mentioned above. In the second part we provide a more constructive characterization of the solutions of the first group, with the purpose of proving some additional properties, such as the nondegeneracy and the asymptotic behaviour as $\lambda\to0$.

Before stating precisely our results, let us start with some observations.  For $\l<1/e$ the equation \eqref{eq:intro_v_lambda} has two constant solutions $\underline{v}_\l<1<\overline{v}_\l$. We let $\mu:=\overline{v}_\l$, so that 
\begin{equation}
\l e^\mu=\mu, \quad \mu>1
\end{equation}
and $\mu\to+\infty$ as $\l\to0$. In order to write the problem in a form more suitable for the bifurcation analysis, we consider the following normalization
\begin{equation}
u:=\frac{v_\l}{\overline{v}_\l}=\frac{v_\l}{\mu}.
\end{equation}
Then problem \eqref{eq:intro_v_lambda} becomes
\begin{equation}\label{eq:intro_u_mu}
\begin{cases}
-\Delta u+u=e^{\mu(u-1)} \quad &\text{in } B_1 \\
\partial_\nu u=0 \quad &\text{on } \partial B_1 \\
u>0 \quad &\text{in } B_1,
\end{cases}
\end{equation}
for $\mu>1$. The equation in this form has the constant solution $u\equiv1$ for every $\mu$. We denote by $\underline{u}_\mu$ the other constant solution, which is characterized by
\begin{equation}\label{eq:underline_u_mu_def}
\underline{u}_\mu=e^{\mu(\underline{u}_\mu-1)}, \quad \underline{u}_\mu<1.
\end{equation}
We are now in position to state our bifurcation result.

\begin{theorem}\label{thm:bifurcation}
For every $i\geq 2$, $(\l_i^{rad},1)$ is a bifurcation point for problem \eqref{eq:intro_u_mu}. Let $\Bcal_i$ be the continuum that branches out of $(\l_i^{rad},1)$. The following holds
\begin{itemize}
\item[(i)] the branches $\Bcal_i$ are unbounded and do not intersect; close to $(\l_i^{rad},1)$, $\Bcal_i$ is a $C^1$-curve;
\item[(ii)] if $u_\mu \in \Bcal_i$ then $u_\mu >0$;
\item[(iii)] each branch consists of two connected components: the component $\Bcal_i^-$, along which $u_\mu(0)<1$, and the component $\Bcal_i^+$, along which $u_\mu(0)>1$;
\item[(iv)] if $u_\mu \in \Bcal_i$ then $u_\mu-1$ has exactly $i-1$ zeros, $u_\mu'$ has exactly $i-2$ zeros and each zero of $u_\mu'$ lies between two zeros of $u_\mu-1$;
\item[(v)] the functions satisfying $u_\mu(0)<1$ are uniformly bounded in the $C^1$-norm.
\end{itemize}
\end{theorem}

In particular, the functions in $\Bcal_2^-$ are monotone increasing, and they share this property with the solutions constructed in \cite{PistoiaVaira2015}. The functions in $\Bcal_i^-$, for $i\geq3$, satisfy $u(0)<1$, $u''(0)>0$ and oscillate around the constant solution. Up to our knowledge, solutions of this type do not appear in the preexisting literature. In the second part of the paper we will produce,   in a more contructive way, solutions having the same qualitative oscillatory behaviour.

The solutions along $\Bcal_2^+$ are monotone decreasing, as the solutions found in 
\cite{MR2209293} in the case $N=2$. Decreasing solutions of \eqref{eq:intro_v_lambda} in dimension $N\geq3$ never appeard in the literature before, as well as solutions which satisfy $u(0)>1$ and oscillate around the origin. It is an interesting open problem to find solutions of this type by a more explicit constructive approach and to detect their asymptotic behaviour, as well as to obtain more information about the bifurcation branches $\Bcal_i^+$ in dimension $N\geq3$.


Concerning the bifurcation branches, we obtain some additional properties, depending on the dimension.

\begin{theorem}\label{thm:bifurcationNgeq3}
If $N\geq3$, the bifurcation point $(\l_i^{rad},1)$ is transcritical; on the right branch we have $u_\mu(0)<1$, on the left branch we have $u_\mu(0)>1$.
\end{theorem}

The previous property is based on \cite[Lemma 3.2]{BonheureGrumiauTroestler2015}, which is proved in dimension $N\geq 3$. The authors conjecture, based on numerical simulations, that it should hold also in dimension $2$. 
The following characteristic instead is a feature of the dimension 2.

\begin{theorem}\label{thm:bifurcationN2}
Let $N=2$. For every $\bar{\mu}$ there exists $C>0$ such that any solution $u_\mu$ found in Theorem \ref{thm:bifurcation} satisfies
\begin{equation}
\|u_\mu\|_{L^\infty(B_1)} \leq C \quad\text{ for } 1<\mu\leq\bar{\mu}.
\end{equation}
\end{theorem}

The second part of this paper is devoted to provide, in a more contructive way, solutions having the same oscillatory behaviour as the ones in $\Bcal_i^-$, $i\geq 2$. First, we build a monotone decreasing  solution by solving a min-max problem. The variational characterization provides more information than the bifurcation approach. In particular, it allows to prove that the monotone solution is nondegenerate as $\mu\to+\infty$ and to analyze its asymptotic behaviour, which is the same as for the solutions in \cite{PistoiaVaira2015}. We also identify the asymptotic behaviour of the oscillating solutions.
In order to do that, we follow closely the method introduced in \cite{BonheureGrossiNorisTerracini2015} where the authors considered the Lin-Ni-Takagi equation when $p\rightarrow \infty$. Before stating precisely our result, let us introduce some notation. Let $G(r,s)$, $s\in (0,1)$, denote the Green function associated to the one dimensional operator
\begin{equation}\label{eq:L_mathcal}
{\mathcal L}: u\mapsto-u''-\frac{N-1}ru'+u,
\end{equation}
for the boundary conditions $u'(0)=u'(1)=0$, that is to say
$$ \mathcal{L} G (\cdot,s)=\delta_s \text{ in } (0,1),\quad \dfrac{\partial G}{\partial r}(0,s)=\dfrac{\partial G}{\partial r}(1,s)=0.$$

\begin{theorem}\label{thm:variational}
Let $k>0$ be an integer. 
\begin{itemize}
\item[(i)] There exists $\bar{\mu}_1(k)$ such that for any $\mu>\bar{\mu}_1(k)$ problem \eqref{eq:intro_u_mu} admits a radial solution having exactly $k$ interior maximum points $0<\a_{1,\mu}<\ldots<\a_{k-1,\mu}<\a_{k,\mu}<1$;
\item[(ii)] There exists $\bar{\mu}_2(k)$ such that for any $\mu>\bar{\mu}_2(k)$ problem \eqref{eq:intro_u_mu} admits a radial solution having exactly $k$ maximum points $0<\a_{1,\mu}<\ldots<\a_{k,\mu}=1$;
\item[(iii)] $(\a_{1,\mu},\ldots,\a_{k,\mu}) \to (\a_1,\ldots,\a_k)$ as $\mu\to\infty$ and $(\a_1,\ldots,\a_k)$ is a critical point of the function
\begin{equation}\label{eq:varphi_def}
\varphi(s_1,\ldots,s_k)=\inf \{ \|u\|_{H^1(B_1)}^2: \, u\in H^1_{rad}(B_1), \, u(s_1)=\ldots=u(s_k)=1 \},
\end{equation}
in the set $0<s_1<\ldots<s_k\leq1$;
\item[(iv)] the solution converges pointwise to $\sum_{j=1}^k A_j G(r,\a_j)$, where $(A_1,\ldots,A_k)$ is a solution of the system 
\end{itemize}
\begin{equation}\label{eq:system_A_j_def}
\sum_{j=1}^k A_j G(\a_i,\a_j)=1, \quad i=1,..,k.
\end{equation}
\end{theorem}

Let us notice that the limit profile is the same as the one obtained for the Lin-Ni-Takagi equation in \cite{BonheureGrossiNorisTerracini2015}.

\begin{remark}\label{rem:half_picks}
Consider problem \eqref{eq:intro_u_mu} in an annulus $B_b\setminus B_a$, $a>0$. In addition to the previous solutions, there also exist solutions having a boundary maximum point at $a$. More precisely, for every integer $k>0$ there exist $\bar{\mu}_3(k)$ and $\bar{\mu}_4(k)$ such that 
\begin{itemize}
\item[(i)] for any $\mu>\bar{\mu}_3(k)$ there exists a radial solution having exactly $k$ maximum points $a=\a_{1,\mu}<\ldots<\a_{k-1,\mu}<\a_{k,\mu}<b$;
\item[(ii)] for any $\mu>\bar{\mu}_4(k)$ there exists a radial solution having exactly $k$ maximum points $a=\a_{1,\mu}<\ldots<\a_{k,\mu}=b$.
\end{itemize}
The analogous of points (iii) and (iv) of Theorem \ref{thm:variational} holds.
\end{remark}

\begin{remark}
When $N=2$, it is possible to provide, for every $\mu>1$, a monotone decreasing solution of \eqref{eq:intro_u_mu} of mountain pass type. This family of solutions shares the monotone behaviour of the solutions along the bifurcation branch $\Bcal_2^+$, as defined in Theorem \ref{thm:bifurcation} (iii). For a sketch of the proof see Remark \ref{rem:dimension2} ahead. A variational characterization of these solutions in dimension higher than 2, as well as a more explicit contruction of oscillatory solutions with $u_\mu(0)>1$ in any dimension, is an interesting open problem.
\end{remark}

The paper is organized as follows. In Section \ref{sec:a_priori_bounds} we collect some a priori bounds for the solutions of \eqref{eq:intro_u_mu}, uniform in the parameter $\mu$. In Section \ref{sec:bifurcation} we perform the bifurcation analysis and we prove Theorems \ref{thm:bifurcation}, \ref{thm:bifurcationNgeq3} and \ref{thm:bifurcationN2}. 

In the remaining sections we prove Theorem \ref{thm:variational}. More precisely, in Section \ref{subsec:existence_increasing} we show that there exists an increasing radial solution of \eqref{eq:intro_u_mu}, characterized as being a mountain pass solution in the cone of nonnegative, nondecreasing functions. This corresponds to the one described in Theorem \ref{thm:variational} (ii) for $k=1$. In Section \ref{subsec:asymptotic_increasing} we analyse its asymptotic behaviour as $\mu\to+\infty$. In Section \ref{subsect:nondegeneracy} we prove, following closely the arguments in \cite[Theorem 5.1]{BonheureGrossiNorisTerracini2015}, that the incresing solution is nondegenerate. This implies, in particular, that it depends in a regular way on the boundary of the domain, as showed in Section \ref{subsec:regular_dependence}. In Section \ref{subsec:decreasing_annulus} we briefly sketch the existence of a decreasing solution in an annulus (which is the one described in Remark \ref{rem:half_picks} (i) for $k=1$) which enjoys similar properties to the increasing one.

In Section \ref{subsec:1layer} we show the existence and convergence of a solution with one interior maximum point, thus proving Theorem \ref{thm:variational} (i) in the case $k=1$. In Section \ref{subsec:convergenceL} we obtain some improved estimates and convergence results, which allow us to conclude the proofs of the remaining main results in Section \ref{subsec:klayer}.

Finally, in the Appendix, we prove some properties of the Green function $G(r,s)$ introduced in \eqref{eq:L_mathcal} in the dimension $N=2$, thus completing \cite[Appendix]{Catrina2009} and \cite[Proposition 2.1]{BonheureGrossiNorisTerracini2015}, which treat the case $N\geq3$.

\section{A priori bounds}\label{sec:a_priori_bounds}

In this section we collect some a priori bounds (uniform in the parameter $\mu$) for the solutions of \eqref{eq:intro_u_mu}.

\begin{lemma}\label{lemma:Jensen}
There exists $C>0$ independent of $\mu$ such that every solution $u$ of \eqref{eq:intro_u_mu} satisfies
\begin{equation}
\|u\|_{L^1(B_1)} + \|e^{\mu(u-1)}\|_{L^1(B_1)} + \|u\|_{W^{2,1}(B_1)}\leq C.
\end{equation}
\end{lemma}
\begin{proof}
We integrate equation \eqref{eq:intro_u_mu} in $B_1$ and we apply Jensen's inequality to obtain
\begin{equation}
\frac{1}{|B_1|}\int_{B_1} u \,dx= \frac{1}{|B_1|} \int_{B_1} e^{\mu(u-1)} \,dx 
\geq e^{\frac{\mu}{|B_1|}\int_{B_1}(u-1)\,dx}.
\end{equation}
Therefore $\bar{u}:=\int_{B_1} u \,dx/|B_1| $ satisfies $\bar{u}-e^{\mu(\bar{u}-1)}\geq0$. Notice that, by the definition of $\underline{u}_\mu$ in \eqref{eq:underline_u_mu_def}, we have that
\begin{equation}\label{eq:x-exp(x-1)}
h(x):=x-e^{\mu(x-1)}>0 \text{ if and only if } x\in (\underline{u}_\mu,1).
\end{equation}
This provides the $L^1$-bound on $u$ and on $e^{\mu(u-1)}$. The $W^{2,1}$-bound follows from standard elliptic regularity.
\end{proof}

The additional assumption $u(0)<1$ provides uniform bounds in the $C^1$-norm.

\begin{lemma}\label{lemma:a_priori_bounds}
There exists $C>0$ independent of $\mu$ such that every solution $u$ of \eqref{eq:intro_u_mu} with $u(0)<1$ satisfies
\begin{equation}\label{eq:a_priori_bounds}
\|u\|_{C^1(\overline{B_1})}\leq C.
\end{equation}
\end{lemma}
\begin{proof}
From the radial equation we get
\[
u''u'-uu'+e^{\mu(u-1)}u' =-\frac{N-1}{r}(u')^2.
\]
Hence the function
\begin{equation}
L(r):=\frac{u'(r)^2}{2}-\frac{u(r)^2}{2}+\frac{e^{\mu(u(r)-1)}}{\mu}
\end{equation}
is non-increasing. If $u(0)<1$, we obtain
\[
L(r)\leq L(0)=-\frac{u(0)^2}{2}+\frac{e^{\mu(u(0)-1)}}{\mu} \leq \frac{1}{\mu} \quad\text{for every } r \in [0,1].
\]
By plotting the curves $L(r)=C$, $C\leq 1/\mu$, in the $(u,u')$-plane, we see that 
\begin{equation}
|u'|\leq 1, \qquad |u|\leq M(\mu) \text{ with } M(\mu)\searrow1\text{ as } \mu\to \infty.
\end{equation}
This provides a $C^1$-bound.
\end{proof}

\begin{remark}\label{remark:boundary_bounds}
By adapting the previous proof we also obtain the following generalization. Let $u_\mu$ be solutions of \eqref{eq:intro_u_mu}. If for every $\mu$ there exists $r_\mu$ such that $u_\mu(r_\mu)<1$ and $u'_\mu(r_\mu)=0$, then 
\begin{equation}
\|u_\mu\|_{C^1(\overline{B_1\setminus B_r})} \leq C,
\end{equation}
with $r=\limsup_{\mu\to\infty} r_\mu$.
\end{remark}

\subsection{A priori bounds in dimension 2}

Another particular situation in which additional a priori bounds hold, is the case of the dimension 2.

\begin{theorem}\label{thm:bounds_dim2}
Let $(\mu_n)\subset \R^+$ be an increasing sequence of positive numbers such that $\mu_n \rightarrow \bar{\mu}<\infty$. Given $c>0$, any sequence $u_n$ of solutions to 
\begin{equation}
\begin{cases}
-\Delta u_n+u_n=e^{\mu_n(u_n-1)} \quad &\text{in } B_1 \subset \R^2 \\
\partial_\nu u_n=0, \ u_n\leq c \quad &\text{on } \partial B_1 \\
u>0 \quad &\text{in } B_1,
\end{cases}
\end{equation}
satisfies
$$\left\|u_n\right\|_{L^\infty (B_1)}\leq C,$$
for some constant $C$ not depending on $n$.
\end{theorem}

The proof follows closely that of \cite[Theorem 3]{BrezisMerle1991} (see also \cite{WangWei2002}). We divide it in several steps.

\begin{lemma}\label{lemma:brezismerle}
Assume that $\Omega \subset \R^2$ be a bounded domain and let $u$ be a solution of
\begin{equation}
\begin{cases}
-\Delta u +u= f(x) &\text{ in } \Omega \\ 
u=0 &\text{ on } \partial \Omega,
\end{cases}
\end{equation}
with $f\in L^1 (\Omega )$. For every $\delta \in (0,4\pi)$, there exists a constant $C$ depending on $\delta$ and $diam (\Omega)$ such that
$$\int_\Omega \exp \left[\dfrac{(4\pi - \delta) |u(x)|}{\left\|f\right\|_{L^1 (\Omega)}} \right] dx \leq C.$$ 
\end{lemma}
\begin{proof}
The proof proceeds exactly as that of \cite[Theorem 1]{BrezisMerle1991}. The only difference is that, instead of working with the Green function of the Laplace operator, we work with the following Green function $-\Delta \mathcal{G} (.,y)+\mathcal{G}(.,y)=\delta_y $, $\mathcal{G}(.,y)=0$ on $\partial \Omega$. The important property is that the two have the same asymptotic behavior, that is to say $\mathcal{G}(x,y)=-\dfrac{1}{2\pi}\log |x-y|+O(1)$ as $|x-y|\to0$. 
\end{proof}

Denote by $S$ the \emph{blow-up set} of $u_n$, that is to say
$$
S=\left\{x\in \overline{B}_1 :\ \exists\, x_n \rightarrow x\ s.t.\ u_n(x_n)\rightarrow +\infty  \right\}.
$$
Notice that by assumption $S\cap \partial B_1=\emptyset$. We aim to prove that $S=\emptyset$.

Since, by Lemma \ref{lemma:Jensen}, $\int_{B_1}  e^{\mu_n (u_n-1)} <C$, there exists a positive bounded Borel measure $\rho$ such that
\begin{equation}\label{eq:rho_measure_def}
\mu_n\int_{B_1} e^{\mu_n (u_n-1)}\varphi \,dx\rightarrow \int_{B_1} \varphi \,d\rho,
\quad \text{for every }\varphi \in C_0^\infty (B_1).
\end{equation}
We say that $x_0\in \overline{B}_1$ is a \emph{regular point} if there exists a function $\varphi  \in C_0^\infty (B_1)$ such that $0\leq \varphi \leq 1$, $\varphi=1$ on a neighborhood of $x_0$ and
$$
\int_{B_1} \varphi \, d\rho <4\pi.
$$
Finally, we define the \emph{set of singular points} $\Sigma$ as the complementary of the set of regular points. 

\begin{lemma}\label{lemma:Sigma}
\begin{itemize}
\item[(i)] If $x_0$ is a regular point, then $(u_n)$ is bounded in $L^\infty (B_{R_0}(x_0)\cap \overline{B}_1)$ for some $R_0>0$.
\item[(ii)] $S=\Sigma$.
\end{itemize}
\end{lemma}
\begin{proof}
(i) Let $x_0$ be a regular point. Since by assumption $u_n$ is bounded on $\partial B_1$, we only consider the case $x_0\in B_1$. By definition, there exists $R_1>0$ such that
\begin{equation}
\mu_n\int_{B_{R_1}(x_0)} e^{\mu_n (u_n-1)} \,dx < 4\pi \quad\text{for every } n.
\end{equation}
We decompose $u_n$ as $u_n=v_n+w_n$ where
$$
\begin{cases}
-\Delta v_n +v_n = e^{\mu_n (u_n-1)} & \text{ in } B_{R_1}(x_0) \\  
v_n=0 & \text{ on } \partial B_{R_1}(x_0),
\end{cases}
$$
and
$$
\begin{cases}
-\Delta w_n +w_n = 0 & \text{ in } B_{R_1}(x_0) \\  
w_n=u_n & \text{ on } \partial B_{R_1}(x_0).
\end{cases}
$$
By Lemma \ref{lemma:Jensen}, there exists $C>0$ independent of $n$ such that $\left\|v_n\right\|_{L^1 (B_{R_1}(x_0))}\leq C$. This, together with the Harnack inequality and Lemma \ref{lemma:Jensen}, provides
$$
\left\|w_n\right\|_{L^\infty(B_{R_1/2}(x_0))}\leq \left\|w_n\right\|_{L^1 (B_{R_1}(x_0))}
\leq C\int_{B_{R_1}(x_0)} (u_n +|v_n|)\,dx \leq C.
$$
Using Lemma \ref{lemma:brezismerle}, we have, for some $\varepsilon>0$,
$$
\int_{B_{R_1}(x_0)} e^{(1+\varepsilon)\mu_n |v_n|} \,dx\leq C.
$$
From the two previous estimates and standard elliptic estimates, we deduce that there exists $R_0$ such that $\left\|u_n\right\|_{L^\infty (B_{R_0}(x_0))}\leq C$. 

(ii) Thanks to the first claim, we have the inclusion $S\subset \Sigma$. Let us prove the reverse inclusion. Let $x_0\in \Sigma$ and suppose by contradiction that there exists $R_0$ such that $\left\|u_n\right\|_{L^\infty (B_{R_0}(x_0))}<C $. Since $0<\mu_n\leq \bar{\mu}$, we have that $e^{\mu_n (u_n -1)}\leq C$ in $B_{R_0}(x_0)$ for any $n$. Therefore, taking a smaller $R_0$ if necessary,
$$
\mu_n\int_{B_{R_0}(x_0)}e^{\mu_n (u_n -1)} \,dx \leq C R_0^2 < 4\pi .
$$
This contradicts the fact that $x_0\in \Sigma$ and establishes the second claim.
\end{proof}

\begin{proof}[Proof of Theorem \ref{thm:bounds_dim2}]
By contradiction suppose that there exists $x_0 \in S$. Since the $u_n$ are uniformly bounded on $\partial B_1$ by assumption, then $x_0\in B_1$. Take $R>0$ small enough such that $S\cap B_R (x_0) = \{x_0\}$. This can be done because, by Lemma \ref{lemma:Jensen} and by the definition of $\rho$ \eqref{eq:rho_measure_def}, $\int_{B_1}d\rho$ is bounded, which implies that $\Sigma =S$ is finite.
Let $z_n$ be a sequence of solutions to 
$$
\begin{cases}
-\Delta z_n +z_n = e^{\mu_n(u_n-1)} & \text{ in } B_R (x_0)\\ 
z_n=0& \text{ on } \partial B_R (x_0).
\end{cases}
$$
By the maximum principle, we have
\begin{equation}\label{eq:z_n_bounded}
\int_{B_R (x_0)}e^{\mu_n z_n}\,dx\leq \int_{B_R (x_0)}e^{\mu_n u_n}\,dx \leq C.
\end{equation}
On the other hand, $z_n\rightarrow z$ a.e. where $z$ solves
$$
\begin{cases}
-\Delta z +z = \dfrac{\rho}{\bar{\mu}} & \text{ in } B_R (x_0)\\ 
z=0& \text{ on } \partial B_R (x_0).
\end{cases}
$$
Since $x_0\in S$, Lemma \ref{lemma:Sigma} (ii) implies that $\rho(\{x_0\})\geq4\pi$, so that $\rho\geq 4\pi \delta_{x_0}$. We deduce that
$$
z\geq \dfrac{4\pi}{2\pi \bar{\mu}} \log \dfrac{1}{|x- x_0|}+O(1)
$$
as $x\to x_0$.
Thus, this yields to
$$
\int_{B_R (x_0)}e^{\mu_n z_n}\,dx \geq C\int_0^R r^{-1}\,dr=\infty,
$$
which contradicts \eqref{eq:z_n_bounded}.
\end{proof}

\begin{remark}
The a priori bounds hold true up to the boundary of the domain, without assuming $u_n\leq c$ on $\partial B_1$. This can be proved by locally rectifying the boundary, as done in \cite[Lemma 3.2]{WangWei2002}. For the reader's convenience, we have preferred to present here only the interior estimates because they are sufficient for our purposes.
\end{remark}

\section{Bifurcation analysis}\label{sec:bifurcation}

Recall that $\l_i^{rad}$ denotes the $i$-th eigenvalue of the operator $-\Delta+Id$ in $B_1$ with Neumann boundary conditions, restricted to the radial functions. Correspondingly, $\varphi_i$ is the associated eigenfunction, normalized in the $L^2$-norm.

\begin{proof}[Proof of Theorem \ref{thm:bifurcation}]
We apply the Crandall-Rabinowitz theorem \cite{CrandallRabinowitz1971} in the space
\begin{equation}
X=\{ u\in C^{2,\a}(\overline{B}_1): \, u \text{ is radial }, \partial_r u(1)=0 \},
\end{equation}
$\a \in(0,1)$. The following operator is well defined in $X$
\begin{equation}
F(\mu,u)=(-\Delta +Id)u-e^{\mu(u-1)},
\end{equation}
with values in 
\begin{equation}
Y=\{ u\in C^{0,\a}(\overline{B}_1): \, u \text{ is radial} \},
\end{equation}
together with its derivatives
\begin{equation}
\partial_u F(\mu,u)[\varphi]=(-\Delta+Id)\varphi-\mu e^{\mu(u-1)}\varphi
\end{equation}
\begin{equation}\label{eq:F_derivative_u_u}
\partial^2_u F(\mu,u)[\varphi,\psi]=-\mu^2 e^{\mu(u-1)}\varphi\psi
\end{equation}
\begin{equation}
\partial^2_{u,\mu} F(\mu,u)[\varphi]=-e^{\mu(u-1)}\varphi-\mu e^{\mu(u-1)} (u-1)\varphi.
\end{equation}
We have that
\begin{equation}
\Ker(\partial_u F(\l_i^{rad},1))=\spann\{\varphi_i\},
\end{equation}
\begin{equation}\label{eq:Psi_def}
\Rk(\partial_u F(\l_i^{rad},1))=\Ker(\Psi), \quad\text{where}\quad \langle\Psi,f\rangle=\int_{B_1} f\varphi_i\,dx,
\end{equation}
for $f\in Y$. Since
\begin{equation}
a:=\langle \Psi,\partial^2_{u,\mu} F(\l_i^{rad},1)[\varphi_i]\rangle=\int_{B_1} \varphi_i^2\,dx=-1,
\end{equation}
the Crandall-Rabinowitz theorem implies that $(\l_i^{rad},1)$ is a bifurcation point.

(i) Again by the Crandall-Rabinowitz theorem there exist a neighborhood $\mathcal{U} \subset \R\times X$ with $(\lambda_i^{rad},1)\in\mathcal{U}$, $\eps>0$ and a $C^1$-curve $\gamma:(-\eps,\eps)\to\R\times X$ with $\gamma(0)=(\lambda_i^{rad},1)$, such that
\begin{equation}\label{eq:bifurcation_curve}
F(\mu,u)=0, \ (\mu,u)\in \mathcal{U} \text{ if and only if there exists } s\in (-\eps,\eps) \text{ s.t. }
(\mu,u)=\gamma(s).
\end{equation}
We let $\gamma(s)=(\gamma_1(s),\gamma_2(s))$ with $\gamma_1\in\R$ and $\gamma_2\in X$.

The fact that the branches are unbounded and do not intersect comes from the fact that the number of zeroes of $u_\mu-1$ is preserved along the branch, as we will show in point (iv) below.

(ii) Let us show that $u_\mu\geq \underline{u}_\mu$. Close to the bifurcation point, $u_\mu$ is close to 1 in the $C^{2,\alpha}$-topology. Suppose that there exist $\mu, u_\mu, r$ such that
\[
u_\mu(r)<\underline{u}_\mu.
\]
We can suppose that $r$ is a minimum point of $u_\mu$, hence we have
\[
u_\mu'(r)=0, \qquad u_\mu''(r)\geq0.
\]
By \eqref{eq:x-exp(x-1)} we have
\[
u_\mu''(r)=u_\mu(r)-e^{\mu(u_\mu(r)-1)} <0,
\]
which is a contradiction.

(iii) By \cite{CrandallRabinowitz1971} we have that, along $\Bcal_i$, the derivative of the curve in \eqref{eq:bifurcation_curve} satisfies
\begin{equation}
\gamma_2'(0)=\varphi_i,
\end{equation}
so that
\begin{equation}
\gamma_2(s)=1+s\varphi_i+o(s) \quad\text{as } s\to0.
\end{equation}
Since $\varphi_i(0)>0$, see equation $(10.7.3)$ of \cite{bessel}, we deduce that $u_\mu(0)<1$ on one connected component, locally near the bifurcation point, and $u_\mu(0)>1$ on the other connected component, locally near the bifurcation point. This property holds along the whole branch because $u_\mu(0)\neq 1$ for every $u_\mu \in \Bcal_i$ and $\mu>\l_i^{rad}$. Indeed, $u_\mu(0)= 1$ would imply $u_\mu\equiv1$ by the local uniqueness of the solution for the Cauchy problem, but this is impossible since $1$ does not belong to the birfucation branch.

(iv) By the theory of Sturm-Liouville, the roots of $u_\mu-1$ are simple and the number of zeros of $u_\mu-1$ remains constant along the branch $\Bcal_i$. In order to prove that this number is $i-1$, let $(\mu_n,u_n)\in \Bcal_i$ be such that $(\mu_n,u_n)\to (\l_i^{rad},1)$ in $\R\times C^{2,\a}(\overline{B}_1)$. The normalized functions
\[
v_n:=\frac{u_n-1}{\|u_n-1\|_{C^{2,\a}(B_1)}}
\]
satisfy
\[
-\Delta v_n+v_n=\frac{e^{\mu_n(u_n-1)}-1}{\|u_n-1\|_{C^{2,\a}(B_1)}} 
= (\l_i^{rad}+o(1)) v_n + \frac{o(\|u_n-1\|_{C^{2,\a}(B_1)})}{\|u_n-1\|_{C^{2,\a}(B_1)}}.
\]
Since $v_n$ is a bounded sequence in $C^{2,\a}(B_1)$, $v_n\to v^*$ in $C^{1,\a}(B_1)$ and we can pass to the limit in the previous equation. We deduce that $v^*=k\varphi_i$ for some $k\neq 0$, so $v^*$ has $i-1$ zeros. Since these zeros are simple, $v_n$ also has $i-1$ zeros for $n$ sufficiently large.

Concerning the zeroes of $u_\mu'$, suppose by contradiction that $u_\mu'(s)=u_\mu'(t)=0$ and, to fix the ideas, that $u_\mu(r)-1<0$ for every $r\in (s,t)$. By point (iii) we have $u_\mu(r) \geq \underline{u}_\mu$, hence \eqref{eq:x-exp(x-1)} implies that 
\[
u(r)-e^{\mu(u(r)-1)}>0 \quad\text{in } (s,t).
\]
On the other hand, by integrating the equation we obtain
\[
0=\int_s^t (u-e^{\mu(u-1)})r^{N-1} \,dr,
\]
which is a contradiction.

(v) By point (iii) the functions on the right branch satisfy $u_\mu(0)<1$, hence Lemma \ref{lemma:a_priori_bounds} applies, providing uniform $C^1$-bounds.
\end{proof}

\begin{proof}[Proof of Theorem \ref{thm:bifurcationNgeq3}]
Using \eqref{eq:F_derivative_u_u} and \eqref{eq:Psi_def} we compute
\begin{equation}
b:=-\frac{1}{2a} \langle\Psi, \partial^2_u F(\l_i^{rad},1)[\varphi_i,\varphi_i] \rangle
=-\frac{1}{2} (\l_i^{rad})^2 \int_{B_1} \varphi_i^3 \,dx.
\end{equation}
It is proved in \cite{BonheureGrumiauTroestler2015} that $b<0$ in dimension $N\geq 3$, which implies that $(\l_i^{rad},1)$ is a transcritical bifurcation point. By \cite{CrandallRabinowitz1971} we also have
\[
u_\mu=1+\frac{\mu-\l_i^{rad}}{b}\varphi_i+o(|\mu-\l_i^{rad}|) \quad\text{as } \mu\to\l_i^{rad}.
\]
Then we can conclude as in point (iii) of Theorem \ref{thm:bifurcation}.
\end{proof}

\begin{proof}[Proof of Theorem \ref{thm:bifurcationN2}]
By point (iv) of Theorem \ref{thm:bifurcation}, every bifurcation solution $u_\mu$ has a point $r_\mu \in [0,1]$ such that $u_\mu(r_\mu)<1$ and $u_\mu'(r_\mu)=0$. Then by Remark \ref{remark:boundary_bounds} the $u_\mu$ are uniformly bounded on $\partial B_1$. 

Theorem \ref{thm:bounds_dim2} applies providing the required bounds.
\end{proof}

\section{Variational characterization of the monotone solutions}\label{sec:variational}

\subsection{Existence of the increasing solution}\label{subsec:existence_increasing}
In this section we prove the existence of an increasing solution of \eqref{eq:intro_u_mu} by a variational method. We conjecture that such solution coincides with the one belonging to the right branch bifurcating from $(\lambda_2^{rad},1)$ that we found in the previous section. It is also possible that it coincides with the one found in \cite{PistoiaVaira2015}.

We work in the more general radial domain $B_b\setminus B_a$, $0\leq a<b$, and we study the following problem
\begin{equation}\label{eq:u_mu_annulus}
\begin{cases}
-\Delta u+u=e^{\mu(u-1)} \quad &\text{in } B_b\setminus B_a \\
\partial_\nu u=0 \quad &\text{on } \partial (B_b\setminus B_a) \\
u>0 \quad &\text{in } B_b\setminus B_a.
\end{cases}
\end{equation}

{\bf Notation.} We use the convention that $B_b\setminus B_0=B_b$, which allows us to treat at the same time the case of the annulus and that of the ball. 
In order to highlight the domain dependence, we denote by $u_\mu(r;a,b)$ a solution of \eqref{eq:u_mu_annulus}. 
When we don't need to put emphasis on the domain dependence, we shall sometimes write more simply $u_\mu(r)$.
The prime signs $u'_\mu(r;a,b)$, $u''_\mu(r;a,b)$, and so on, denote always derivatives with respect to the variable $r$.

Let
\begin{multline}
\Ccal_{+}(a,b)=\{u\in H^1_{rad}(B_b\setminus B_a):\ 0\leq u(r)\leq C \text{ for every } a \leq r \leq b \\ 
u(r)\leq u(s) \text{ for every } a < r\leq s \leq b\},
\end{multline}
with $C$ defined in \eqref{eq:a_priori_bounds}.

\begin{theorem}\label{thm:increasing_sol_existence}
For $\mu>\l_2^{rad}(a,b)$ there exists an increasing radial solution $u_{\mu,+}(r)=u_{\mu,+}(r;a,b)$ of \eqref{eq:u_mu_annulus}, which has the following variational characterization
\begin{equation}\label{eq:zeta_mu_+_def}
z_{\mu,+}(r):=u_{\mu,+}(r)-\underline{u}_\mu \in \Ccal_{+}(a,b) 
\end{equation}
and
\begin{equation}\label{eq:c_mu_+_def}
E_\mu(z_{\mu,+};a,b)=\inf_{\substack{z\in \Ccal_{+}(a,b) \\ z\not\equiv0}} \sup_{t\geq0} E_\mu(tz;a,b)=:c_{\mu,+}(a,b).
\end{equation}
Recall that $\underline{u}_\mu$ was defined in \eqref{eq:underline_u_mu_def}.
Here 
\begin{multline}
E_\mu(z;a,b):= \int_{B_b\setminus B_a} \left(\frac{|\nabla z|^2}{2}+\frac{(z+\underline{u}_\mu)^2}{2} -\frac{e^{\mu(z+\underline{u}_\mu-1)}}{\mu} \right)\,dx\\
=\int_{B_b\setminus B_a} \left(\frac{|\nabla z|^2}{2}+\frac{z^2}{2} +\underline{u}_\mu \left(z-\frac{e^{\mu z}}{\mu} \right) +\frac{\underline{u}_\mu^2}{2} \right)\,dx,
\end{multline}
which is well defined in $\Ccal_+(a,b)$ because of the $L^\infty$-bound inside the definition of $\Ccal_+(a,b)$.
\end{theorem}
\begin{proof}
We perform the change of variables $z=u-\underline{u}_\mu$. Then $u$ solves the equation in \eqref{eq:u_mu_annulus} if and only if $z$ solves
\begin{equation}\label{eq:z_BNW}
-\Delta z+z=f(z) \quad\text{with}\quad f(z):=\underline{u}_\mu(e^{\mu z}-1).
\end{equation}
For $\mu>\l_2^{rad}$, $f$ satisfies the assumptions of \cite[Theorem 1.3]{BonheureNorisWeth2012}, apart from the assumption $f'(0)=0$ (in our case we have $f'(0)=\underline{u}_\mu\mu$). Such assumption is used in \cite{BonheureNorisWeth2012} to ensure that the problem has the mountain pass geometry at 0. Once we show that $E_\mu$ has the mountain pass geometry at 0, the proof of \cite[Theorem 1.3]{BonheureNorisWeth2012} applies without changes to our case, thus providing the existence of an increasing radial solution $z_{\mu,+}$ of \eqref{eq:z_BNW}, enjoying the variational characterization \eqref{eq:zeta_mu_+_def}-\eqref{eq:c_mu_+_def}.

In order to prove that $E_\mu$ has the mountain pass geometry at 0, we introduce the following version of the Nehari manifold
\begin{equation}\label{eq:nehari_def}
\Ncal:=\left\{ z\in\Ccal_+(a,b) : \ z\not\equiv0, \, \int_{B_b\setminus B_a} \left( |\nabla z|^2+z^2- f(z)z \right) \, dx=0 \right\}.
\end{equation}
This set was first used in \cite{SerraTilli2011}.
For  $z\in \Ccal_{+}(a,b)$ we also let $g(t):=E_\mu(tz;a,b)$, $t\geq0$. For every $\mu>1$, the following holds.

(i) For every $z\in \Ccal_{+}(a,b)$, $g(t)$ has at least one positive maximum point. Indeed, notice that $g'(0)=0$ and that
\[
g''(0)=\int_{B_b\setminus B_a} (|\nabla z|^2+(1-\mu \underline{u}_\mu)z^2) \,dx.
\]
We deduce from \eqref{eq:underline_u_mu_def} and \eqref{eq:x-exp(x-1)} that 
\begin{equation}\label{eq:underline_u_mu_asymptotics}
h'(\underline{u}_\mu)=1-\mu e^{\mu(\underline{u}_\mu-1)}=1-\mu \underline{u}_\mu >0,
\end{equation}
hence $g$ has a strict local minimum at zero. On the other hand, $g$ diverges to $-\infty$ as $t\to +\infty$, which provides the claim.

(ii) For every $z\in \Ccal_{+}(a,b)$, $g(t)$ has exactly one maximum point $t(z,\mu)>0$. This comes from the facts that  the function $f(z)/z$ is monotone increasing.

(iii) $\inf_\Ncal \|z\|_{H^1(B_b\setminus B_a)}>0$. Suppose by contradiction that there exists a sequence $\{z_n\}\subset\Ncal$ such that $\|z_n\|_{H^1(B_b\setminus B_a)}\to0$. If $a>0$ we immediately have that $\|z_n\|_{L^\infty(B_b\setminus B_a)}\to0$ by the continuity of the embedding $H^1(B_b\setminus B_a)\hookrightarrow L^\infty(B_b\setminus B_a)$ for $a>0$. If $a=0$ the same conclusion holds thanks to the fact that the $z_n$ are positive and non-decreasing: we have
\[
\|z_n\|_{L^\infty(B_b)}= \|z_n\|_{L^\infty(B_b \setminus B_{b/2})} \leq C 
\|z_n\|_{H^1(B_b \setminus B_{b/2})} \leq C \|z_n\|_{H^1(B_b)} \to0.
\]
Therefore in both cases we have
\[
e^{\mu z_n}-1 < (\mu+\eps)z_n 
\]
for every $\eps>0$ and for $n$ sufficiently large. Then the definition of $\Ncal$ provides
\[
\int_{B_b\setminus B_a} (|\nabla z_n|^2+(1-\underline{u}_\mu(\mu+\eps))z_n^2)\,dx <0,
\]
which contradicts \eqref{eq:underline_u_mu_asymptotics} provided that $\eps<(1-\mu\underline{u}_\mu)/\underline{u}_\mu$.
\end{proof}

\begin{remark}\label{rem:existence_sol}
For a fix $\bar{\mu}$, if $0<\bar{a}<\bar{b}$ are such that there exists the solution $u_{\bar{\mu},+}(\cdot;\bar{a},\bar{b})$, then by the continuity of $\lambda_2^{rad}(a,b)$, there exist $0<A_1<\bar{a}<A_2$, $B_1<\bar{b}<B_2$ such that the solution $u_{\mu,+}(\cdot;a,b)$ exists for every $(a,b)\in(A_1,A_2)\times(B_1,B_2)$ and $\mu\geq\bar{\mu}$. In case $\bar{a}=0$, there exist $B_1<\bar{b}<B_2$ such that the analogous holds in the ball.
\end{remark}

\begin{remark}\label{rem:dimension2}
When $N=2$, it is possible to provide, for every $\mu>1$, a monotone decreasing solution of \eqref{eq:intro_u_mu} of mountain pass type. This can be done proceeding similarly to the proof of Theorem \ref{thm:increasing_sol_existence}, with the only difference of working in the cone of nonnegative nonincreasing solutions instead of $\Ccal_{+}(a,b)$. Since $N=2$, $E_\mu$ satisfies the Palais-Smale condition, hence no a priori estimates are necessary in this case.
Indeed, the following two conditions hold for $h(u)=\underline{u}_\mu (e^{\mu u}-1)$ :
\begin{enumerate}
\item $\dfrac{h(z)}{z}\rightarrow \infty$ when $z\rightarrow \infty$, there exist $a_1$ and a function $f(z)$ satisfying $\dfrac{f(z)}{z^2}\rightarrow 0$ when $z\rightarrow$ such that
$$h(z)\leq a_1 e^{f(z)},\ z\geq 0.$$
\item 
Let $H(z)=\int_0^z h(t) dt$. There exist $a_2$ and $\theta \in (0,1/2) $ such that $H(z)\leq \theta z h(z)$ if $z\geq a_2$.
\end{enumerate}
Then it follows from \cite{MR0370183} that $E_\mu$ satisfies the Palais-Smale condition.
\end{remark}

\subsection{Asymptotic behaviour of the increasing solution}\label{subsec:asymptotic_increasing}
Let $G(r,s;a,b)$ be the Green function associated to the operator
\[
\mathcal{L}: u\mapsto -u''-\frac{N-1}{r}u'+u
\]
for the boundary conditions $u'(a)=u'(b)=0$, that is to say
\begin{equation}\label{eq:Green_def}
\mathcal{L} G(\cdot,s;a,b) =\delta_s \text{ in } (a,b), \quad 
\frac{\partial G}{\partial r} (a,s;a,b)=\frac{\partial G}{\partial r} (b,s;a,b)=0.
\end{equation}

The punctual limit of $G(r,s;a,b)$ as $s\to b$ is well defined and we denote it by $G(r,b;a,b)$. Analogously, if $a>0$, the punctual limit of $G(r,s;a,b)$ as $s\to a$ is well defined and we denote it by $G(r,a;a,b)$. Moreover we have that
\begin{equation}\label{eq:green_increasing}
G(r,b;a,b) \text{ is monotone increasing}, 
\end{equation}
\begin{equation}
G(r,a;a,b) \text{ ($a>0$) is monotone decreasing}.
\end{equation}
For a proof of these facts see for example \cite[Proposition 2.2]{BonheureGrossiNorisTerracini2015}

\begin{theorem}\label{thm:asymptotic_increasing_sol}
Let $u_{\mu,+}(r;a,b)$ be the increasing solution found in Theorem \ref{thm:increasing_sol_existence}.
As $\mu\to\infty$ we have that $u_{\mu,+}(\cdot;a,b)\to G(\cdot,b;a,b)$ in $H^1(B_b\setminus B_a)\cap C^{0,\gamma}(\overline{B_b\setminus B_a})$ for every $\gamma\in(0,1)$.
\end{theorem}

In order to prove the theorem we need some preliminary lemmas.

\begin{lemma}
There exists $u_{\infty,+} \in \Ccal_+(a,b)$ satisfying $u_{\infty,+}(b)=1$ such that, up to a subsequence, we have
\begin{equation}\label{eq:u_mu_convergence}
u_{\mu,+} \rightharpoonup u_{\infty,+} \text{ in } H^1(B_b\setminus B_a), \qquad
u_{\mu,+}\to u_{\infty,+} \text{ in } C^{0,\gamma}(\overline{B_b\setminus B_a}) 
\end{equation}
for every $\gamma\in(0,1)$, as $\mu\to+\infty$.
\end{lemma}
\begin{proof}
We integrate the equation in \eqref{eq:u_mu_annulus} in $B_b\setminus B_a$ to obtain
\begin{equation}\label{eq:u_mu_integrated}
\int_{B_b\setminus B_a} (u_{\mu,+}-e^{\mu(u_{\mu,+}-1)}) \, dx=0
\end{equation}
Recalling that $u_{\mu,+}$  is increasing and that $u_{\mu,+}\geq\underline{u}_\mu$, relation \eqref{eq:x-exp(x-1)} implies that
\begin{equation}\label{eq:u_mu<1>1}
u_{\mu,+}(a)<1, \quad u_{\mu,+}(b)>1 \quad \text{for every } \mu.
\end{equation}
Then Lemma \ref{lemma:a_priori_bounds} applies. From the $C^1$-bounds therein and the compactness of the embedding $C^1\hookrightarrow C^{0,\gamma}$ for every $\gamma\in(0,1)$, we deduce that there exists $u_{\infty,+}$ such that \eqref{eq:u_mu_convergence} holds.
By the pointwise convergence we have that $u_{\infty,+}\in \Ccal_+(a,b)$.

Let us prove that $u_{\infty,+}(b)=1$. From \eqref{eq:u_mu<1>1} we have $u_{\infty,+}(b)\geq1$. Suppose by contradiction that $u_{\infty,+}(b)>1$. Then there exist $s\in(a,b)$ and $\delta>0$ such that $u_{\mu,+}(r)>1+\delta$ for $r\in (s,b)$. By integrating the radial equation in $(s,b)$ we obtain
\begin{equation}
s^{N-1}u_{\mu,+}'(s) =\int_s^1 (e^{\mu(u_{\mu,+}-1)}-u_{\mu,+}) \,dr \to +\infty
\end{equation}
as $\mu\to+\infty$. This contradicts the a priori bounds in Lemma \ref{lemma:a_priori_bounds}, hence we deduce that $u_{\infty,+}(b)=1$.
\end{proof}

It only remains to prove that the limit function above coincides with $G$. This is what we will do in the following.

\begin{lemma}\label{lemma:gamma_convergence_first_ineq}
Let
\begin{equation}
c_{\infty,+}(a,b)=\inf \left\{ \frac{\|z\|^2_{H^1(B_b\setminus B_a)}}{2}  : \, z\in \Ccal_{+}(a,b), \, z(b)=1 \right\}.
\end{equation}
Then $c_{\infty,+}(a,b)\leq \liminf_{\mu\to+\infty} c_{\mu,+}(a,b)$.
\end{lemma}
\begin{proof}
As $\underline{u}_\mu\to0$, the sequence $z_{\mu,+}$ defined in \eqref{eq:zeta_mu_+_def} also converges to the function $u_{\infty,+}$ introduced in the previous lemma. Using that $u_{\infty,+}(b)=1$, we have 
\begin{equation}\label{eq:gamma_convergence1}
c_{\infty,+}(a,b) \leq \frac{\|u_{\infty,+}\|_{H^1(B_b\setminus B_a)}^2}{2} 
\leq \liminf_{\mu\to+\infty} \frac{\|z_{\mu,+}\|_{H^1(B_b\setminus B_a)}^2}{2}.
\end{equation}
On the other hand, we have
\begin{equation}\label{eq:gamma_convergence2}
\begin{split}
\frac{\|z_{\mu,+}\|_{H^1(B_b\setminus B_a)}^2}{2}
&=E_\mu(z_{\mu,+};a,b)+ \underline{u}_\mu \int_{B_b\setminus B_a} \left( \frac{e^{\mu z_{\mu,+}}}{\mu} -z_{\mu,+}- \frac{\underline{u}_\mu}{2}\right) \,dx \\
&=c_{\mu,+}(a,b) + \underline{u}_\mu \int_{B_b\setminus B_a}  \frac{e^{\mu z_{\mu,+}}}{\mu} \,dx +o(\mu),
\end{split}
\end{equation}
where we used the a priori bounds in Lemma \ref{lemma:a_priori_bounds}. Relation \eqref{eq:u_mu_integrated} provides
\begin{equation}\label{eq:gamma_convergence3}
\underline{u}_\mu \int_{B_b\setminus B_a}  \frac{e^{\mu z_{\mu,+}}}{\mu} \,dx
=\int_{B_b\setminus B_a} \frac{z_{\mu,+}+\underline{u}_\mu}{\mu} \,dx =o(\mu).
\end{equation}
By combining \eqref{eq:gamma_convergence1}, \eqref{eq:gamma_convergence2} and \eqref{eq:gamma_convergence3} we obtain the claim.
\end{proof}

\begin{lemma}\label{lemma:nehari_bounded_uniformly}
The Nehari set $\Ncal$ introduced in \eqref{eq:nehari_def} is bounded uniformly in $\mu$ and bounded away from zero unformly in $\mu$.
\end{lemma}
\begin{proof}
It is sufficient to adapt the arguments of the proof of Theorem \ref{thm:increasing_sol_existence}, taking into account the dependence on $\mu\to\infty$.
\end{proof}

\begin{lemma}\label{lemma:underline_u_mu_limit}
We have that $\lim_{\mu\to+\infty} e^{\mu \underline{u}_\mu}=1$.
\end{lemma}
\begin{proof}
By combining \eqref{eq:underline_u_mu_def} and \eqref{eq:underline_u_mu_asymptotics} we find
\[
\lim_{\mu\to+\infty} e^{\mu} \underline{u}_\mu = \lim_{\mu\to+\infty} e^{\mu \underline{u}_\mu} \in [0,e].
\]
Then
 \[
\lim_{\mu\to+\infty} \mu \underline{u}_\mu = \lim_{\mu\to+\infty} \mu e^{-\mu} e^{\mu} \underline{u}_\mu =0,
\]
which provides the statement.
\end{proof}

\begin{proof}[Proof of Theorem \ref{thm:asymptotic_increasing_sol}]
Let for the moment $G(r)=G(r,b;a,b)/G(b,b;a,b)$, with $G(r,b;a,b)$ defined below \eqref{eq:Green_def}. Denote by $t(G,\mu)$ the coefficient projecting $G$ onto $\Ncal$, as in point (ii) of the proof of Theorem \ref{thm:increasing_sol_existence}. We claim that
\begin{equation}\label{eq:t_G_mu_to1}
\lim_{\mu\to\infty} t(G,\mu)=1.
\end{equation}
If $t(G,\mu)<1$ eventually as $\mu\to\infty$, then by \eqref{eq:green_increasing} $t(G,\mu)G <1$ eventually in $B_b\setminus B_a$ and the condition $t(G,\mu)G\in\Ncal$ provides
\begin{equation}\label{eq:gamma_convergence4}
t \int_{B_b\setminus B_a} (|\nabla G|^2+G^2)\,dx
=\int_{B_b\setminus B_a} (e^{\mu(tG-1)}e^{\mu\underline{u}_\mu}-\underline{u}_\mu)G \, dx \to 0 
\end{equation}
as $\mu\to\infty$, since $e^{\mu\underline{u}_\mu}\to1$ by Lemma \ref{lemma:underline_u_mu_limit}. This contradicts the fact that $\Ncal$ is bounded away from zero unformly in $\mu$, as claimed in Lemma \ref{lemma:nehari_bounded_uniformly}. Similarly, if $t(G,\mu)$ is eventually larger than 1, then $t(G,\mu)G >1$ in a set of positive measure. Hence the right hand side in \eqref{eq:gamma_convergence4} diverges as $\mu\to\infty$, contradicting the fact that $\Ncal$ is bounded uniformly in $\mu$.

Therefore we have proved \eqref{eq:t_G_mu_to1}, which implies
\begin{equation}
t(G,\mu)G\to G \text{ in } H^1(B_b\setminus B_a), \qquad t(G,\mu)G\in \Ncal.
\end{equation}
As proved in \cite[Proposition 4.1]{BonheureGrossiNorisTerracini2015} $G$ achieves $c_{\infty,+}(a,b)$, hence we have
\begin{multline}\label{eq:gamma_convergence_second_ineq}
c_{\infty,+}(a,b) =\frac{\|G\|_{H^1(B_b\setminus B_a)}^2}{2} 
=\lim_{\mu\to\infty} \frac{\|t(G,\mu)G\|_{H^1(B_b\setminus B_a)}^2}{2} \\
=\lim_{\mu\to\infty} \left\{ E_\mu(tG;a,b)+\underline{u}_\mu \int_{B_b\setminus B_a}\left(\frac{e^{\mu t G}}{\mu}-G-\frac{\underline{u}_\mu}{2} \right)\,dx\right\}
\geq \limsup_{\mu\to\infty} c_{\mu,+}(a,b).
\end{multline}
In the last step we used the fact that $t(G,\mu)G$ is an admissible test function in the minimization problem \eqref{eq:c_mu_+_def}, and the following estimate
\begin{multline}
\underline{u}_\mu \int_{B_b\setminus B_a} \frac{e^{\mu t G}}{\mu} \,dx 
\leq \frac{\underline{u}_\mu}{\mu \min_{B_b\setminus B_a}G} \int_{B_b\setminus B_a} e^{\mu tG} G \,dx \\
=\frac{1}{\mu \min_{B_b\setminus B_a}G} \int_{B_b\setminus B_a} (t(|\nabla G|^2+G^2)+\underline{u}_\mu G)\,dx \to0,
\end{multline}
as $\mu\to\infty$.

By combining \eqref{eq:gamma_convergence_second_ineq} with Lemma \ref{lemma:gamma_convergence_first_ineq} we obtain that $c_{\infty,+}(a,b)=\lim_{\mu\to\infty}c_{\mu,+}(a,b)$, which in turn implies the statement.
\end{proof}

Let
\begin{equation}\label{eq:u_infty+def}
u_{\infty,+}(r;a,b):=G(r,b;a,b).
\end{equation}
As a consequence of the previous proof we also obtain
\begin{equation}
\lim_{\mu\to\infty} c_{\mu,+}(a,b) =c_{\infty,+}(a,b) =\frac{\|u_{\infty,+}\|^2_{H^1(B_b\setminus B_a)}}{2} 
=\frac{|\partial B_b|}{2} u_{\infty,+}'(b;a,b)
\end{equation}
(to obtain the last equality integrate by parts the equation satisfied by $u_{\infty,+}$). 
Moreover, by standard elliptic regularity theory, the convergence of $u_{\mu,+}$ to $u_{\infty,+}$ is $C^\infty$ on the set where $u_{\infty,+}$ is strictly less than 1, that is to say
\begin{equation}\label{eq:convergence_strong_u_mu}
u_{\mu,+} \to u_{\infty,+} \text{ in } C^{\infty}(\overline{B_{b-\varepsilon}\setminus B_a}), \text{ for every } \varepsilon>0.
\end{equation}

By combining the convergence with the Pohozaev identity, we also deduce the following estimate.

\begin{lemma}\label{lemma:pohozaev}
We have
\begin{equation}\label{eq:pohozaev}
\lim_{\mu\to\infty} \frac{e^{\mu(u_{\mu,+}(b;a,b)-1)}}{\mu}=\frac{(u_{\infty,+}'(b;a,b))^2}{2}.
\end{equation}
\end{lemma}
\begin{proof}
The Pohozaev identity for $u_{\mu,+}$ provides
\begin{multline}
|\partial B_b| \frac{e^{\mu(u_{\mu,+}(b)-1)}}{\mu} = -\frac{N-2}{2} \int_{B_b\setminus B_a} |\nabla u_{\mu,+}|^2\,dx -\frac{N}{2} \int_{B_b\setminus B_a} u_{\mu,+}^2 \,dx \\
+\frac{N}{\mu} \int_{B_b\setminus B_a} u_{\mu,+} \,dx +\int_{\partial(B_b\setminus B_a)} \frac{u_{\mu,+}^2}{2} \,d\sigma,
\end{multline}
while the Pohozaev identity for $u_{\infty,+}$ gives
\begin{multline}
|\partial B_b| \frac{u_{\infty,+}'(b)^2}{2} = -\frac{N-2}{2} \int_{B_b\setminus B_a} |\nabla u_{\infty,+}|^2\,dx -\frac{N}{2} \int_{B_b\setminus B_a} u_{\infty,+}^2 \,dx \\
+\int_{\partial(B_b\setminus B_a)} \frac{u_{\infty,+}^2}{2} \,d\sigma.
\end{multline}
The convergence proved in Theorem \ref{thm:asymptotic_increasing_sol} provides the assertion.
\end{proof}

\subsection{Non-degeneracy of the increasing solution}\label{subsect:nondegeneracy}

\begin{theorem}\label{thm:non_degeneracy}
Let $v_\mu$ solve
\begin{equation}\label{eq:v_mu}
\begin{cases}
-\Delta v_\mu+v_\mu=\mu e^{\mu(u_{\mu,+}-1)} v_\mu \quad &\text{in } B_b\setminus B_a \\
\partial_\nu v_\mu=0 \quad &\text{on } \partial (B_b\setminus B_a).
\end{cases}
\end{equation}
For $\mu$ sufficiently large we have $v_\mu\equiv0$.
\end{theorem}

The proof of this theorem follows very closely that of \cite[Theorem 5.1]{BonheureGrossiNorisTerracini2015}, therefore we only highlight the main differences.

\begin{lemma}\label{lemma:u_tilde_mu_convergence}
Let
\begin{equation}
\tilde{u}_\mu(r):=\mu \left[ u_{\mu,+}\left( b+\frac{r}{k\mu} \right)-u_{\mu,+}(b) \right], 
\qquad r\in [-(b-a)k\mu,0],
\end{equation}
with
\begin{equation}\label{eq:k_def}
k:=\frac{u_{\infty,+}'(b;a,b)}{\sqrt{2}}.
\end{equation}
Then
\begin{equation}\label{eq:u_tilde_convergence}
\tilde{u}_\mu(r)\to \tilde{u}_\infty(r):=\log\frac{4e^{\sqrt{2}r}}{(1+e^{\sqrt{2}r})^2} \quad \text{in } C^1_{loc}(-\infty,0).
 \end{equation}
\end{lemma}
\begin{proof}
Notice first that for every $R>0$ there exists $C>0$ independent of $\mu$ such that
\begin{equation}
|\tilde{u}_\mu(r)|+|\tilde{u}_\mu'(r)|\leq C, \qquad r\in (-R,0).
\end{equation}
Moreover, $\tilde{u}_\mu$ solves the following equation
\begin{equation}
\begin{cases}
-\tilde{u}_\mu''-\frac{N-1}{b+\frac{r}{k\mu}} \frac{\tilde{u}_\mu'}{k\mu} +\frac{\tilde{u}_\mu}{k^2\mu^2}+\frac{u_{\mu,+}(b)}{k^2\mu} = \frac{e^{\mu(u_{\mu,+}(b)-1)}}{k^2\mu}e^{\tilde{u}_\mu}, \quad r\in [-(b-a)k\mu,0] \\
\tilde{u}_\mu(0)=\tilde{u}_\mu'(0)=0.
\end{cases}
\end{equation}
The right hand side of the previous equation is bounded because, thanks to Lemma \ref{lemma:pohozaev}, we have
\begin{equation}\label{eq:pohozaev_consequence}
\frac{e^{\mu(u_{\mu,+}(b)-1)}}{k^2\mu} \to1.
\end{equation}
We conclude that also $\tilde{u}_\mu''$ is bounded in $(-R,R)$, hence $\tilde{u}_\mu$ converges to some function $u$ in $C^1_{loc}(-\infty,0)$, which satisfies
\begin{equation}
\begin{cases}
-u''=e^u \quad r\in [-\infty,0]\\
u(0)= u'(0)=0.
\end{cases}
\end{equation}
Then it is proved in \cite[Lemma 4.2]{Grossi2006} that $u$ has the form in \eqref{eq:u_tilde_convergence}.
\end{proof}

\begin{lemma}\label{lemma:linearized_rescaling}
Let $v_\mu$ be a nontrivial solution of \eqref{eq:v_mu} and let
\begin{equation}
\tilde{v}_\mu(r):= \frac{v_{\mu}\left( b+\frac{r}{k\mu} \right)}{\|v_\mu\|_{L^\infty(B_b\setminus B_a)}}, 
\qquad r\in [-(b-a)k\mu,0].
\end{equation}
Then $\tilde{v}_\mu\to0$ in $C^1_{loc}(-\infty,0)$.
\end{lemma}
\begin{proof}
Notice first that there exists $C>0$ independent of $\mu$ such that
\begin{equation}\label{eq:v_tilde_bounds}
|\tilde{v}_\mu(r)|+|\tilde{v}_\mu'(r)|\leq C, \qquad r\in \left(-\frac{(b-a)k\mu}{2},0\right).
\end{equation}
Moreover, $\tilde{v}_\mu$ solves the following equation
\begin{equation}
\begin{cases}
-\tilde{v}_\mu''-\frac{N-1}{b+\frac{r}{k\mu}} \frac{\tilde{v}_\mu'}{k\mu} +\frac{\tilde{v}_\mu}{k^2\mu^2} 
= \frac{e^{\mu(u_{\mu,+}(b)-1)}}{k^2\mu}e^{\tilde{u}_\mu} \tilde{v}_\mu, \quad r\in (-(b-a)k\mu,0) \\
\tilde{v}'_\mu(0)=0 \\
\|\tilde{v}_\mu\|_{L^\infty}=1.
\end{cases}
\end{equation}
From \eqref{eq:u_tilde_convergence}, \eqref{eq:pohozaev_consequence} and \eqref{eq:v_tilde_bounds} we deduce that $\tilde{v}_\mu\to \tilde{v}_\infty$ in $C^1_{loc}(-\infty,0)$, with
\begin{equation}
\begin{cases}
-\tilde{v}_\infty'' =\frac{4e^{\sqrt{2}r}}{(1+e^{\sqrt{2}r})^2} \tilde{v}_\infty, \quad r\in (-\infty,0) \\
\tilde{v}'_\infty(0)=0 \\
\|\tilde{v}_\infty\|_{L^\infty}\leq1.
\end{cases}
\end{equation}
Then \cite[Lemma 4.2]{Grossi2006} implies that $\tilde{v}_\infty\equiv0$.
\end{proof}

\begin{lemma}\label{lemma:non_deg_lebesgue}
There exist $\bar{\mu}>1$, $\bar{s}\in (-(b-a)k\bar{\mu}/2,0)$ and $C>0$ such that
\begin{equation}
\tilde{u}_\mu'(s) \geq C
\end{equation}
for every $\mu\geq\bar{\mu}$, $s \in (-(b-a)k\mu/2,\bar{s})$.
\end{lemma}
\begin{proof}
We integrate the equation $-(r^{N-1}u_{\mu,+}')'+r^{N-1}u_{\mu,+} = e^{\mu(u_{\mu,+}-1)}r^{N-1}$ between $a$ and $b+s/(k\mu)$, with
\[
s \in \left( -\frac{(b-a)k\mu}{2},0 \right),
\]
to obtain
\begin{multline}\label{eq:l12mu}
\left(b+\frac{s}{k\mu}\right)^{N-1} k \tilde{u}_{\mu}'(s) 
= \left(b+\frac{s}{k\mu}\right)^{N-1} u_{\mu,+}'\left(b+\frac{s}{k\mu}\right) \\
=\int_a^{b+\frac{s}{k\mu}} u_{\mu,+} (t)\,dt -\int_a^{b+\frac{s}{k\mu}} e^{\mu(u_{\mu,+}-1)}t^{N-1} \,dt 
=:I_{1,\mu}(s)-I_{2,\mu}(s).
\end{multline}
On the one hand, Theorem \ref{thm:asymptotic_increasing_sol} implies that
\begin{equation}\label{eq:I1mu}
I_{1,\mu}(s) \geq \frac{1}{2} \int_a^{\frac{a+b}{2}} G(t,b;a,b) \,dt \geq C,
\end{equation}
for a constant $C>0$ not depending on $\mu$ and $s$.
On the other hand, by the change of variables $\tau=k\mu(t-b)$, we have
\begin{multline*}
I_{2,\mu}(s) = \int_{-(b-a)k\mu}^s e^{\tilde{u}_\mu(\tau)} \frac{e^{\mu(u_{\mu,+}(b)-1)}}{k\mu} \left(b+\frac{\tau}{k\mu}\right)^{N-1} \,d\tau \\
\to k b^{N-1} \int_{-\infty}^{s} e^{\tilde{u}_\infty(\tau)} \,d\tau \qquad \text{as } \mu\to+\infty,
\end{multline*}
where in the second line we used \eqref{eq:u_tilde_convergence}, \eqref{eq:pohozaev_consequence} and the Lebesgue dominated convergence theorem. Since $ \int_{-\infty}^{0} e^{\tilde{u}_\infty(\tau)} \,d\tau<\infty$, for every $\delta>0$ there exists $s(\delta)<0$, $\mu(\delta)>0$ such that 
\begin{equation}\label{eq:I2mu}
I_{2,\mu}(s)<\delta \quad \text{for every } \mu>\mu(\delta) \text{ and } s\in (-(b-a)k\mu,s(\delta)).
\end{equation}
By combining \eqref{eq:l12mu}, \eqref{eq:I1mu} and \eqref{eq:I2mu} we obtain the statement.
\end{proof}

\begin{proof}[End of the proof of Theorem \ref{thm:non_degeneracy}]
Using Lemmas \ref{lemma:linearized_rescaling}, \ref{lemma:non_deg_lebesgue} and proceeding exactly as in the proof of \cite[Theorem 5.1]{BonheureGrossiNorisTerracini2015} Step 3, it is possible to show that any nontrivial solution $v_\mu$ of \eqref{eq:v_mu} satisfies
\begin{equation}\label{eq:contradiction_step3}
\frac{v_\mu(r)}{\|v_\mu\|_{L^\infty(B_b\setminus B_a)}}
= G(r,b;a,b) k\mu \int_{-\frac{b-a}{2}k\mu}^0 e^{\tilde{u}_\mu(t)} \tilde{v}_\mu(t) \,dt +o_\mu(1).
\end{equation}
For every $\mu$ let $r_\mu\in [a,b]$ be such that $v_\mu(r_\mu)=\|v_\mu\|_{L^\infty(B_b\setminus B_a)}$. Then \eqref{eq:contradiction_step3} provides
\begin{equation}\label{eq:contradiction_step3_2}
1= G(r_\mu,b;a,b) k\mu \int_{-\frac{b-a}{2}k\mu}^0 e^{\tilde{u}_\mu(t)} \tilde{v}_\mu(t) \,dt +o_\mu(1).
\end{equation}
On the other hand, Lemma \ref{lemma:linearized_rescaling} implies that $v_\mu(b)\to0$ as $\mu\to\infty$, so that \eqref{eq:contradiction_step3} gives
\begin{equation}
o_\mu(1)= G(b,b;a,b) k\mu \int_{-\frac{b-a}{2}k\mu}^0 e^{\tilde{u}_\mu(t)} \tilde{v}_\mu(t) \,dt ,
\end{equation}
which contradicts \eqref{eq:contradiction_step3_2}.
\end{proof}

\subsection{Uniqueness of the increasing solution}\label{subsect:uniqueness_increasing}

Exploiting the proof of Theorem \ref{thm:non_degeneracy}, one can also prove that the mountain pass value $c_{\mu,+}(a,b)$ in \eqref{eq:c_mu_+_def} is uniquely achieved (for more details see \cite[Theorem 5.1 and Corollary 5.3]{BonheureGrossiNorisTerracini2015}).

\begin{theorem}\label{thm:uniqueness_minimal_energy_sol}
There exists $\bar{\mu}(a,b)$ such that, for $\mu>\bar{\mu}(a,b)$, the value $c_{\mu,+}(a,b)$ is uniquely achieved by a multiple of $u_{\mu,+}(\cdot;a,b)$. In addition, one can choose the value of $\bar{\mu}$ valid for an open neighbourhood of $a$ and $b$.
\end{theorem}

\subsection{Regular dependence on the boundary points}\label{subsec:regular_dependence}

\begin{lemma}\label{lemma:continuous_dependence}
Let $\mu$ sufficiently large be fixed. For $0<A_1<A_2<B_1<B_2$ as in Remark \ref{rem:existence_sol}, define
\begin{equation}\label{eq:I_def}
I=\left\{ (r,a,b): \, A_1<a<A_2, \, B_1<b< B_2, \, a<r<b \right\}.
\end{equation}
Then the map $ I\ni (r,a,b) \mapsto u_{\mu,+}(r;a,b)$ is continuous. 

Similarly, in the case of the ball, let $0<B_1<B_2$ be as in Remark \ref{rem:existence_sol} and $I=\left\{ (r,b): \, B_1<b< B_2, \, 0\leq r<b \right\}$. Then the map $ I\ni (r,b) \mapsto u_{\mu,+}(r;0,b)$ is continuous.
\end{lemma}
\begin{proof}
We prove the result in the case of the annulus, the case of the ball being analogous.
Let $(r,a_n,b_n)$ be a sequence in $I$ such that $a_n\to a_*$, $b_n\to b_*$. 
In the following $\mu$ is fixed and we consider sequences in $n$, hence we denote
\begin{equation}
u_n(r):=u_{\mu,+}(r;a_n,b_n), \quad u_*(r):=u_{\mu,+}(r;a_*,b_*).
\end{equation}
Let also $\hat{u}_n$, $\hat{u}_*$ be the trivial extensions of $u_n$, $u_*$ in the interval $[A,B]:=[A_1,B_2]$ (extend as a constant outside $(a_n,b_n)$). 

Since $\{ \hat{u}_n  \}$ is bounded in $H^1(B_{B}\setminus B_{A})$, there exists $\tilde{u}\in H^1(B_{B}\setminus B_{A})$ such that (up to a subsequence)
\begin{equation*}
\hat{u}_n \rightharpoonup \tilde{u} \quad \text{weakly in } H^1(B_{B}\setminus B_{A}).
\end{equation*}
We have to prove that $\tilde{u}\equiv u_*$.

By the pointwise convergence, $\tilde{u}$ is non-decreasing and $\tilde{u}\geq \underline{u}_\mu$. 
Let $\varphi\in C_c^\infty(B_{\b_*}\setminus B_{\a_*})$, then $\varphi\in C_c^\infty(B_{b_n}\setminus B_{a_n})$ for $n$ sufficiently large and the $H^1$-weak convergence implies
\begin{equation*}
\int_{B_{b_*}\setminus B_{a_*}} \left( \nabla\tilde{u}\cdot\nabla\varphi +\tilde{u}\varphi\right) \,dx 
= \int_{B_{\b_*}\setminus B_{\a_*}} e^{\mu(\tilde{u}-1)} \varphi \,dx.
\end{equation*}
Therefore both $\tilde{u}$ and $u_*$ solve equation \eqref{eq:u_mu_annulus} in $B_{\b_*}\setminus B_{\a_*}$. In particular, $\tilde{u}$ is bounded by Lemma \ref{lemma:a_priori_bounds} and $\tilde{u}-\underline{u}_\mu$ belongs to the Nehari set $\Ncal$ in \eqref{eq:nehari_def}. Therefore $\tilde{u}$ can be used as a test function for $c_{\mu,+}(\a_*,\b_*)$ and the uniqueness result in Theorem \ref{thm:uniqueness_minimal_energy_sol} provides
\begin{equation}
E_\mu(u_*-\underline{u}_\mu;a_*,b_*)<E_\mu(\tilde{u}-\underline{u}_\mu;a_*,b_*).
\end{equation}
On the other hand, we have by the $H^1$-convergence
\begin{equation}
E_\mu(\tilde{u}-\underline{u}_\mu;a_*,b_*) \leq \liminf_{n\to\infty} E_\mu(\hat{u}_n-\underline{u}_\mu;a_*,b_*).
\end{equation}
We combine the two previous inequlity and the continuity of $E_\mu(\cdot;a,b)$ with respect to $a$ and $b$ to obtain
\begin{multline}
\lim_{n\to \infty} E_\mu(\hat{u}_*-\underline{u}_\mu;a_n,b_n) =E_\mu(u_*-\underline{u}_\mu;a_*,b_*)
< E_\mu(\tilde{u}-\underline{u}_\mu;a_*,b_*) \\
\leq \liminf_{n\to\infty} E_\mu(\hat{u}_n-\underline{u}_\mu;a_*,b_*) 
= \leq \liminf_{n\to\infty} E_\mu(u_n-\underline{u}_\mu;a_n,b_n).
\end{multline}
This implies that $\hat{u}_*$ achieves $c_{\mu,+}(a_n,b_n)$ for $n$ large, which contradicts Theorem \ref{thm:uniqueness_minimal_energy_sol}.
\end{proof}

\begin{lemma}\label{lemma:C1_dependence}
In the same assumptions of the previous lemma, the maps $ I\ni (r,a,b) \mapsto u_{\mu,+}(r;a,b)$ and $ I\ni (r,b) \mapsto u_{\mu,+}(r;0,b)$ are of class $C^1$.
\end{lemma}
\begin{proof}
The lemma can be proved exactly as in \cite[Lemma 5.8]{BonheureGrossiNorisTerracini2015}, by defining $A=A_1$, $B=B_2$ and 
\begin{equation}
\hat{u}_{\mu,+}(s;a,b)= u_{\mu,+}(hs+k;a,b),
\end{equation}
with
\[
h=\frac{a-b}{A-B} \quad\text{and}\quad k=\frac{Ab-Ba}{A-B}. \qedhere
\]
\end{proof}

\subsection{The decreasing solution in the annulus}\label{subsec:decreasing_annulus}

In this section we consider problem \eqref{eq:u_mu_annulus} in an annulus $B_b\setminus B_a$ with $a>0$. In this case, the a priori bounds on the solutions come from the continuity of the embedding $W^{1,p}_{rad}(B_b\setminus B_a)\hookrightarrow L^\infty(B_b\setminus B_a)$. More precisely, we have the following.

\begin{lemma}\label{lemma:a_priori_bounds_annulus}
Let $a>0$. There exists $C>0$ independent of $\mu$ such that every solution $u$ of \eqref{eq:u_mu_annulus} satisfies
\begin{equation}\label{eq:a_priori_bounds_annulus}
\|u\|_{C^1(B_b\setminus B_a)} \leq C.
\end{equation}
\end{lemma}
\begin{proof}
From the uniform bounds in Lemma \ref{lemma:Jensen} and from the continuity of the embedding $W^{1,p}_{rad}(B_b\setminus B_a)\hookrightarrow L^\infty(B_b\setminus B_a)$, we deduce that $u$ is uniformly bounded in the $L^\infty$-norm. By integrating \eqref{eq:u_mu_annulus} and using again Lemma \ref{lemma:Jensen} , we obtain
\begin{equation*}
|u'(r)| \leq \frac{1}{r^{N-1}} \int_a^r (u+e^{\mu(u-1)}) s^{N-1} \,ds,
\end{equation*}
and hence the $C^1$-bound since $r\geq a>0$.
\end{proof}

Hence we can prove the existence of a decreasing solution by working in the set
\begin{multline}
\Ccal_{-}(a,b)=\{u\in H^1_{rad}(B_b\setminus B_a):\ 0\leq u(r)\leq C \text{ for every } a \leq r \leq b \\ 
u(r)\geq u(s) \text{ for every } a < r\leq s \leq b\}.
\end{multline}

The decreasing solution in the annulus has the same properties as the increasing solution, which we state without proof.

\begin{theorem}\label{thm:decreasing_sol_existence}
Let $a>0$. For $\mu>\l_2^{rad}(a,b)$ there exists a decreasing radial solution $u_{\mu,-}(r)=u_{\mu,-}(r;a,b)$ of \eqref{eq:u_mu_annulus}, which has the following variational characterization
\begin{equation}\label{eq:zeta_mu_-_def}
z_{\mu,-}(r):=u_{\mu,-}(r)-\underline{u}_\mu \in \Ccal_{-}(a,b) 
\end{equation}
and
\begin{equation}\label{eq:c_mu_-_def}
E_\mu(z_{\mu,-};a,b)=\inf_{\substack{z\in \Ccal_{-}(a,b) \\ z\not\equiv0}} \sup_{t\geq0} E_\mu(tz;a,b)=:c_{\mu,-}(a,b).
\end{equation}

As $\mu\to\infty$ we have that
\begin{equation}
u_{\mu,-}(\cdot;a,b)\to u_{\infty,-}(\cdot;a,b):=G(\cdot,a;a,b)
\end{equation}
in $H^1(B_b\setminus B_a)\cap C^{0,\gamma}(\overline{B_b\setminus B_a})\cap C^{\infty}(\overline{B_{b}\setminus B_{a+\varepsilon}})$ for every $\gamma\in(0,1)$ and $\eps>0$.

Moreover
\begin{equation}
\lim_{\mu\to\infty} c_{\mu,-}(a,b) =\frac{|\partial B_a|}{2} u_{\infty,-}'(a;a,b).
\end{equation}
and
\begin{equation}
\lim_{\mu\to\infty} \frac{e^{\mu(u_{\mu,-}(a;a,b)-1)}}{\mu}=\frac{(u_{\infty,-}'(a;a,b))^2}{2}.
\end{equation}

For $\mu$ sufficiently large, $u_{\mu,-}(\cdot;a,b)$ is non-degenerate and $c_{\mu,-}(a,b)$ is uniquely achieved by a multiple of $u_{\mu,-}(\cdot;a,b)$. As a consequence, the map $I \ni (r,a,b) \mapsto u_{\mu,-}(r;a,b)$ is of class $C^1$ for $\mu$ sufficiently large, with $I$ defined in \eqref{eq:I_def}.
\end{theorem}

\section{Existence and convergence of the k-layer solutions}\label{sec:existence_convergence_k_layer}

\subsection{The 1-layer solution}\label{subsec:1layer} We are in a position to prove the existence and convergence of a solution with one interior maximum point.

\begin{theorem}\label{BGNT6.1}
For $\mu$ sufficiently large there exists a radial solution $u_{\mu,1-layer}(r;a,b)$ of \eqref{eq:u_mu_annulus} having exactly one maximum point at $r=\bar{s}_\mu(a,b)$. Furthermore
\begin{equation}
\bar{s}_\mu(a,b)\to \bar{s}_\infty=\bar{s}_\infty(a,b), \qquad 
u_{\mu,1-layer}(r;a,b)\to \frac{G(r,\bar{s}_\infty;a,b)}{G(\bar{s}_\infty,\bar{s}_\infty;a,b)}
\end{equation}
as $\mu\to+\infty$. The point $\bar{s}_\infty$ lies in the interior of the interval $(a,b)$ and it is the unique point which satisfies
\begin{equation}\label{eq:s_bar_def}
\left. \left( \frac{G(r,r;a,b)}{r^{N-1}} \right)'\right|_{r=\bar{s}_\infty}=0.
\end{equation}
\end{theorem}
\begin{proof}
For $s\in(a,b)$ let
\begin{equation}
\label{defLmu}
L_\mu(s;a,b):=\frac{e^{\mu(u_{\mu,+}(s;a,s)-1)} - e^{\mu(u_{\mu,-}(s;s,b)-1)} }{\mu}.
\end{equation}
By Lemma \ref{lemma:pohozaev} we have that $L_\mu(\cdot;a,b)$ converges pointwise to
\begin{equation}
\label{defLinf}
L_\infty(s;a,b):=\frac{(u_{\infty,+}'(s;a,s))^2-(u_{\infty,-}'(s;s,b))^2}{2}.
\end{equation}
When $N\geq 3$, it is proved in \cite[Lemma 2.4,Theorem 6.1]{BonheureGrossiNorisTerracini2015} that $L_\infty(\cdot;a,b)$ is strictly increasing and has a unique interior zero $\bar{s}_\infty(a,b)$ and that $\bar{s}_\infty$ satisfies \eqref{eq:s_bar_def}. In case $N=2$, the proof can be repeated without changes by making use of the result in Lemma \ref{lemma:appendix1}.

Since $L_\mu(\cdot;a,b)$ is continuous for $\mu$ sufficiently large (by Lemma \ref{lemma:continuous_dependence}), it has a zero $\bar{s}_\mu$. Then
\begin{equation}
u_{\mu,1-layer}(r;a,b):=\begin{cases}
u_{\mu,+}(r;a,\bar{s}_\mu) \quad\text{for } r\in [a,\bar{s}_\mu) \\
u_{\mu,-}(r;\bar{s}_\mu,b) \quad \text{for } r\in [\bar{s}_\mu,b]
\end{cases}
\end{equation}
solves \eqref{eq:u_mu_annulus}.
\end{proof}

\subsection{$C^1$-convergence of $L_\mu$}\label{subsec:convergenceL} In order to construct the $k$-layer solutions, we first prove that the functions $L_\mu$ defined in \eqref{defLmu} converge to $L_\infty$ in the $C^1$-norm.

\begin{lemma}\label{lemma:pohozaev_uniform}
Fix $a\in [0,1)$ and $a<B_1<B_2\leq1$. It holds
\begin{equation}\label{eq:pohozaev_uniform}
\lim_{\mu\to\infty} \sup_{b\in [B_1,B_2]} \left[ \frac{e^{\mu(u_{\mu,+}(b;a,b)-1)}}{\mu}
-\frac{(u_{\infty,+}'(b;a,b))^2}{2}\right]=0.
\end{equation}
\end{lemma}
\begin{proof}
By contradiction, suppose that \eqref{eq:pohozaev_uniform} does not hold.
Thus we can find two sequences $\mu_n \rightarrow +\infty$ and $b_n\rightarrow b \in [B_1,B_2]$ such that
\[
\frac{e^{\mu_n(u_{\mu_n,+}(b_n;a,b_n)-1)}}{\mu_n}
-\frac{(u_{\infty,+}'(b_n;a,b_n))^2}{2} \rightarrow C >0.
\]
Using the smoothness of $u_{\infty,+}$, we have
\begin{equation}\label{eq:contradiction_pohozaev_uniform}
\frac{e^{\mu_n(u_{\mu_n,+}(b_n;a,b_n)-1)}}{\mu_n}
-\frac{(u_{\infty,+}'(b;a,b))^2}{2} \rightarrow C >0.
\end{equation}
By Lemma \ref{lemma:a_priori_bounds} (which holds independently of $b_n$) we have
\begin{equation}\label{eq:a_priori_bounds_b_n}
\sup_n \left( \|u_{\mu_n,+}'(\cdot;a,b_n)\|_{L^\infty(B_{b_n}\setminus B_a)}
+\|u_{\mu_n,+}(\cdot;a,b_n)\|_{L^\infty(B_{b_n}\setminus B_a)} \right) \leq C.
\end{equation}
Let $\hat{u}_{\mu_n,+}(r;a,b_n)$ be the trivial extension of $u_{\mu_n,+}(r;a,b_n)$ in the interval $[a,B_2]$, and analogously for $u_{\infty,+}(r;a,b)$. Then \eqref{eq:a_priori_bounds_b_n} implies
\[
\hat{u}_{\mu_n,+}(\cdot;a,b_n) \to \hat{u}_{\infty,+}(\cdot;a,b) \quad \text{in } C^{0,\gamma}(B_{B_2}\setminus B_a) \text{ and weak-} H^1(B_{B_2}\setminus B_a).
\]
By using the equation satisfied by $u_{\mu_n,+}(r;a,b_n)$ in $B_{b_n}\setminus B_a$, we see that the convergence is also strong in $H^1(B_{b}\setminus B_a)$. Then by repeating the proof of Lemma \ref{lemma:pohozaev} with $b=b_n$ and $\mu=\mu_n$ we obtain a contradiction with \eqref{eq:contradiction_pohozaev_uniform}.
\end{proof}

\begin{lemma}\label{lemma:u_b_derivative_bound}
Fix $a\in [0,1)$ and $a<B_1<B_2\leq1$.
There exists $C>0$ such that
\begin{equation}\label{eq:u_b_derivative_bound}
\sup_{\mu>1} \sup_{b\in[B_1,B_2]} \left\|\dfrac{\partial u_{\mu,+}}{\partial b}(.;a,b)\right\|_{L^\infty(B_b\setminus B_a)} \leq C.
\end{equation}
\end{lemma}
\begin{proof}
We notice that $\dfrac{\partial u_{\mu,+}}{\partial b}(r)=\dfrac{\partial u_{\mu,+}}{\partial b}(r;a,b)$ exists by Lemma \ref{lemma:C1_dependence} and solves
\begin{equation}\label{eq:u_b_equation}
\begin{cases}
-\left(\dfrac{\partial u_{\mu,+}}{\partial b}\right)^{\prime \prime} - \dfrac{N-1}{r}\left(\dfrac{\partial u_{\mu,+}}{\partial b}\right)^\prime +\dfrac{\partial u_{\mu,+}}{\partial b}=\mu e^{\mu (u_{\mu,+}-1)}\dfrac{\partial u_{\mu,+}}{\partial b}, \quad r\in\ (a,b)\\
\left(\dfrac{\partial u_{\mu,+}}{\partial b}\right)^\prime (a)=0 \\
\left(\dfrac{\partial u_{\mu,+}}{\partial b}\right)^\prime (b)=-u^{\prime \prime}_{\mu,+}(b).
\end{cases}
\end{equation}
We set $f_\mu (r)=f_\mu (r;a,b)=\dfrac{\partial u_{\mu,+}}{\partial b}(r;a,b)+u^\prime_{\mu,+}(r;a,b)$. We have
\begin{equation}\label{eq:f_b_equation}
\begin{cases}
-f_\mu^{\prime \prime} - \dfrac{N-1}{r}f_\mu^\prime +f_\mu=\mu e^{\mu (u_{\mu,+}-1)}f_\mu-\dfrac{N-1}{r^2}u^\prime_{\mu,+}, \quad r\in\ (a,b)\\
f_\mu^\prime (a)=u_{\mu,+}^{\prime \prime}(a) \\
f_\mu^\prime (b)=0.
\end{cases}
\end{equation}
Let $\varphi$ be the unique solution of
\begin{equation}
\begin{cases}
-\varphi^{\prime \prime} - \dfrac{N-1}{r}\varphi^\prime +\varphi =0, \quad r\in\ (a,b)\\
\varphi^\prime (a)=u_{\infty,+}(a) \\
\varphi^\prime (b)=0.
\end{cases}
\end{equation}
Since $u_{\mu,+}^{\prime \prime}(a)=u_{\mu ,+}(a)- e^{\mu (u_{\mu,+}(a)-1)}\rightarrow u_{\infty,+}(a)$ as $\mu \rightarrow \infty$, we have
\begin{equation}
\label{14marse2}
f_\mu (r)=\int_a^b G(r,s;a,b) \left(\mu e^{\mu (u_{\mu,+}(s)-1)}f_\mu(s)-\dfrac{N-1}{s^2}u^\prime_{\mu,+}(s) \right)ds+\varphi (r)+o_\mu (1),
\end{equation}
when $\mu \rightarrow \infty$.

Noticing that $u_{\mu,+}^\prime$ is bounded independently of $\mu$ and $b$, to prove the lemma it suffices to show that 
\begin{equation}\label{eq:tilde_f_contradiction}
\sup_{\mu>1} \sup_{b\in[B_1,B_2]} \left\| f_\mu(.;a,b)\right\|_{L^\infty(B_b\setminus B_a)} \leq C,
\end{equation}
for a constant $C>0$ independent of $\mu$ and $b$.
Suppose by contradiction that there exist two sequences $(\mu_n)_n$ and $(b_n)_n$ such that  $\mu_n \rightarrow \infty$, $b_n \in [B_1,B_2] \rightarrow b$ and $\|f_n\|_{L^\infty(B_{b_n}\setminus B_a)}=\|f_{\mu_n}(.;a,b_n) \|_{L^\infty(B_{b_n}\setminus B_a)}\rightarrow \infty$ when $n \rightarrow \infty$.  

We claim that there exists $C>0$ not depending on $n$ and $r$ such that
\begin{equation}\label{14marse1}
\dfrac{|f_n^\prime (r)|}{\|f_n \|_{L^\infty(B_{b_n}\setminus B_a)}}\leq C\mu_n,\text{ for all } r\in \left[\dfrac{a+b_n}{2},b_n\right].
\end{equation}
Indeed, integrating \eqref{eq:f_b_equation}, we see that, for $r\in  [(a+b_n)/2,b_n]$,
\begin{multline}
\dfrac{|f_n^\prime (r)|}{\|f_n\|_{L^\infty(B_{b_n}\setminus B_a)}}r^{N-1}
\leq \int_r^{b_n} \dfrac{|f_n (t)|}{\|f_n \|_{L^\infty(B_{b_n}\setminus B_a)}}t^{N-1}\,dt 
\\+ \frac{1}{\|f_n\|_{L^\infty(B_{b_n}\setminus B_a)}}
\int_r^{b_n} \left( \mu_n e^{\mu_n (u_{\mu_n,+}(t;a,b_n)-1)} |f_n(t)|
+\dfrac{N-1}{t^2} |u^\prime_{\mu_n,+}(t;a,b_n)| \right) \,dt.  
\end{multline}
Since $u^\prime_{\mu_n,+}(.;a,b_n)$ is $L^\infty$-bounded independently of $n$, we have
$$
\dfrac{|f_n^\prime (r)|}{\|f_n \|_{L^\infty(B_{b_n}\setminus B_a)}}\leq C+ C \mu_n \int_r^{b_n} e^{\mu_n (u_{\mu_n,+}(t;a,b_n)-1)}  \,dt.
$$
Then, the claim \eqref{14marse1} follows from the fact that $ \int_r^{b_n} e^{\mu_n (u_{\mu_n,+}(t;a,b_n)-1)}  \,dt$ is uniformly bounded thanks to Lemma \ref{lemma:u_tilde_mu_convergence}.

Now, we define
$$
\tilde{f}_n (r)= \dfrac{f_n (b_n+ \frac{r}{k \mu_n})}{\|f_n \|_{L^\infty(B_{b_n}\setminus B_a)}},
\quad r\in \left[-k\mu_n (b_n-a),0\right],
$$
where $k$ is defined as in \eqref{eq:k_def}. Then, we have
\begin{equation}
\begin{cases}
-\tilde{f}_n^{\prime \prime} - \dfrac{N-1}{b_n+\frac{r}{k \mu_n} }\frac{1}{k \mu_n} \tilde{f}_n^\prime +(\frac{1}{k \mu_n})^2 \tilde{f}_n=(\frac{1}{k \mu_n})^2\mu_n e^{\mu_n (u_{\mu_n,+}(.;a,b_n)-1)}\tilde{f}_n\\
\quad \quad \quad \quad \hspace{2cm}-\dfrac{N-1}{(b+\frac{r}{k \mu_n})^2}(\frac{1}{k \mu_n})^2 \dfrac{u^\prime_{\mu_n,+}(.;a,b_n)}{\|f_n\|_{L^\infty(B_{b_n}\setminus B_a)}}, \quad r\in\ (- k \mu_n (b_n-a),0)\\
\tilde{f}_n^\prime (0)=0 .
\end{cases}
\end{equation}
Notice that here $\mu_n (b_n-a)\to\infty$ because $b_n\geq B_2>a$.
By \eqref{14marse1}, we see that $|\tilde{f}_n^\prime|\leq \mu_n \frac{C}{k \mu_n} \leq C^\prime$. 
Using the blow-up analysis in Lemma \ref{lemma:u_tilde_mu_convergence}, we deduce that $\tilde{f}_n \rightarrow \tilde{f}_\infty$ in $C^1_{loc}(-\infty ,0)$, where $\tilde{f}_\infty$ satisfies
$$\begin{cases}-\tilde{f}_\infty^{\prime \prime}=e^{\tilde{u}_\infty}\tilde{f}_\infty,\text{ in } (-\infty,0)\\ \tilde{f}_\infty^\prime (0)=0,\ |\tilde{f}_\infty|\leq 1.\end{cases}$$
One can show, using the classification result of \cite{Grossi2006}, that
\begin{equation}\label{eq:f_tilde_limit}
\tilde{f}_n \rightarrow \tilde{f}_\infty\equiv 0 \quad\text{in } C^1_{loc}(-\infty ,0).
\end{equation}
This implies, proceeding as in Step 3 of \cite[Theorem 5.1]{BonheureGrossiNorisTerracini2015}, that \eqref{14marse2} can be rewritten as
$$
\frac{f_n (r)}{\|f_n\|_{L^\infty(B_{b_n}\setminus B_a)}}
= G(r,b_n;a,b_n) k\mu \int^0_{-k\mu_n \frac{b_n-a}{2}}   e^{\tilde{u}_{\mu_n,+}(t;a,b_n)}
\tilde{f}_n(t) \,dt  +o_n (1). 
$$
Let $r_n\in[a,B_2]$ be such that $f_n(r_n)=\|f_n\|_{L^\infty(B_{b_n}\setminus B_a)}$
Then, we obtain that
$$
1
= G(r_n,b_n;a,b_n) k\mu \int^0_{-k\mu_n \frac{b_n-a}{2}}   e^{\tilde{u}_{\mu_n,+}(t;a,b_n)}
\tilde{f}_n(t) \,dt  +o_n (1). 
$$
We obtain a contradiction with the fact that $\tilde{f}_n (b_n)=o_n (1)$ (which can be deduced from \eqref{eq:f_tilde_limit}). This provides \eqref{eq:tilde_f_contradiction} and hence concludes the proof.
\end{proof}

\begin{corollary}
Fix $a\in [0,1)$ and $a<B_1<B_2\leq1$. For every $b\in [B_1,B_2]$ there exists a function $C(b)$ such that
\begin{equation}\label{eq:u_b_convergence}
\dfrac{\partial u_{\mu,+}}{\partial b}(r;a,b)\rightarrow C(b) u_{\infty,+}(r;a,b), \quad r\in(a,b),
\end{equation}
pointwise as $\mu \rightarrow \infty$ (recall the definition of $u_{\infty,+}$ in \eqref{eq:u_infty+def}).
\end{corollary}

We recall the following result.

\begin{lemma}[{\cite[Lemma 7.3]{BonheureGrossiNorisTerracini2015}}]\label{lemma:u_infty_integrals}
We have
\begin{equation}\label{eq:u_infty_equality1}
(N-1)\int_a^b u_{\infty,+}' u_{\infty,+} r^{N-3} \,dr=b^{N-1} \left( u_{\infty,+}''(b)-u_{\infty,+}'(b)^2 \right) - a^{N-1} u_{\infty,+}(a)^2,
\end{equation}
\begin{equation}\label{eq:u_infty_equality2}
2\int_a^b u_{\infty,+}^2 r^{N-1} \,dr = b^{N-1} \left( u_{\infty,+}'(b)+b u_{\infty,+}''(b) \right) -b^N u_{\infty,+}'(b)^2 -a^{N} u_{\infty,+}(a)^2.
\end{equation}
\end{lemma}

\begin{lemma}\label{BGNTL7.4}
For $0\leq a<b\leq1$ we have
\begin{equation}\label{eq:mu_u_b}
\lim_{\mu\to\infty} \mu\frac{\partial u_{\mu,+}}{\partial b}(b;a,b)=
2\dfrac{u_{\infty,+}^{\prime \prime}(b;a,b)- (u_{\infty,+}^\prime (b;a,b))^2}{u^\prime_{\infty,+}(b;a,b)}.
\end{equation}
\end{lemma}
\begin{proof}
Let $w:=u_{\mu,+}^\prime$, then
\begin{equation}\label{eq:w}
\begin{cases}
-w^{\prime \prime}-\dfrac{N-1}{r}w^\prime +w = \mu e^{\mu (u_{\mu,+}-1)}w
-\dfrac{N-1}{r^2}w, \quad r\in\ (a,b)\\
w(a)=w(b)=0\\
w^\prime (a)=u_{\mu,+}^{\prime \prime}(a), w^\prime (b)= u_{\mu,+}^{\prime \prime}(b).
\end{cases}
\end{equation}
We multiply equation \eqref{eq:w} by $r^{N-1}\frac{\partial u_{\mu,+}}{\partial b}$ and equation \eqref{eq:u_b_equation} by $r^{N-1}w$; we integrate in $(a,b)$ and subtract the two, to get
\begin{equation}\label{eq:w_integrated}
b^{N-1}u_{\mu,+}^{\prime \prime}(b) \frac{\partial u_{\mu,+}}{\partial b} (b) -a^{N-1} u_{\mu,+}^{\prime \prime}(a) \frac{\partial u_{\mu,+}}{\partial b} (a)= (N-1)\int_a^b u_{\mu,+}^\prime \frac{\partial u_{\mu,+}}{\partial b} r^{N-3}\,dr.
\end{equation}
Let $z:=r u_{\mu,+}^\prime$, then
\begin{equation}\label{eq:z}
\begin{cases}
-z^{\prime \prime}-\dfrac{N-1}{r}z^\prime +z = e^{\mu (u_{\mu,+}-1)}(2+\mu z)-2u_{\mu,+}, \quad r\in\ (a,b)\\
z(a)=z(b)=0\\
z^\prime (a)=a u_{\mu,+}^{\prime \prime}(a), z^\prime (b)=b u_{\mu,+}^{\prime \prime}(b).
\end{cases}
\end{equation}
Proceeding as above we obtain
\begin{equation}\label{eq:z_integrated}
b^{N}u_{\mu,+}^{\prime \prime}(b) \frac{\partial u_{\mu,+}}{\partial b} (b) -a^{N} u_{\mu,+}^{\prime \prime}(a) \frac{\partial u_{\mu,+}}{\partial b} (a)= 2 \int_a^b  \frac{\partial u_{\mu,+}}{\partial b} (u_{\mu,+}-e^{\mu(u_{\mu,+}-1)} )r^{N-1}\,dr.
\end{equation}
Integrating \eqref{eq:u_b_equation} we deduce
\begin{equation}
\int_a^b e^{\mu(u_{\mu,+}-1)} \frac{\partial u_{\mu,+}}{\partial b} r^{N-1} \,dr
=\frac{1}{\mu} \left( \int_a^b \frac{\partial u_{\mu,+}}{\partial b} r^{N-1} \,dr + b^{N-1} u_{\mu,+}''(b) \right),
\end{equation}
hence we can rewrite \eqref{eq:z_integrated} as follows
\begin{multline}\label{eq:z_integrated2}
b^{N}u_{\mu,+}^{\prime \prime}(b) \frac{\partial u_{\mu,+}}{\partial b} (b) -a^{N} u_{\mu,+}^{\prime \prime}(a) \frac{\partial u_{\mu,+}}{\partial b} (a)+\frac{2}{\mu} b^{N-1} u_{\mu,+}''(b) \\
= 2 \int_a^b  \frac{\partial u_{\mu,+}}{\partial b} \left(u_{\mu,+}-\frac{1}{\mu}\right)r^{N-1}\,dr.
\end{multline}
We can pass to the limit in \eqref{eq:w_integrated} and in \eqref{eq:z_integrated2} by means of the estimates \eqref{eq:convergence_strong_u_mu}, \eqref{eq:pohozaev} and  \eqref{eq:u_b_convergence}. We combine the result with Lemma \ref{lemma:u_infty_integrals}, to obtain
\begin{equation*}
-\mu \left( \frac{u_{\infty,+}'(b)^2}{2}+o(1) \right) \frac{\partial u_{\mu,+}}{\partial b}(b) 
=C(b) \left( u_{\infty,+}''(b)-u_{\infty,+}'(b)^2 \right) +o(1)
\end{equation*}
\begin{multline*}
-\mu b \left( \frac{u_{\infty,+}'(b)^2}{2}+o(1) \right) \frac{\partial u_{\mu,+}}{\partial b}(b) - u_{\infty,+}'(b)^2 \\
=C(b) \left( u_{\infty,+}'(b)+b u_{\infty,+}''(b) -b u_{\infty,+}'(b)^2\right)  +o(1),
\end{multline*}
where we also used the fact that $u_{\infty,+}''(a)=u_{\infty,+}(a)$ and 
\begin{equation}
\frac{1}{\mu} \int_a^b u_{\infty,+} r^{N-1}\,dr =o(1).
\end{equation}
From the previous system we get
\begin{equation}\label{eq:C_b}
C(b)=-u_{\infty,+}'(b)
\end{equation}
and \eqref{eq:mu_u_b}.
\end{proof}



\begin{lemma}\label{BGNTL7.6}
Fix $a\in [0,1)$ and $a<B_1<B_2\leq1$. It holds
\begin{equation}
\lim_{\mu\to\infty} \sup_{b\in [B_1,B_2]} \left[ \mu\frac{\partial u_{\mu,+}}{\partial b}(b;a,b)-
2\dfrac{u_{\infty,+}^{\prime \prime}(b;a,b)- (u_{\infty,+}^\prime (b;a,b))^2}{u^\prime_{\infty,+}(b;a,b)} \right] =0
\end{equation}
\end{lemma}
\begin{proof}
By contradiction, suppose that there exist two sequences $\mu_n \rightarrow +\infty$ and $b_n\rightarrow b \in [B_1,B_2]$ such that
\begin{equation}\label{eq:contradiction_BGNTL7.6}
\mu_n \dfrac{\partial u_{\mu_n,+}}{\partial b}(b_n;a,b_n)- 2 \dfrac{u_{\infty,+}^{\prime \prime}(b) - (u_{\infty,+}^\prime (b))^2}{u_{\infty,+}^\prime (b)} \rightarrow C >0.
\end{equation}
Since the bound in Lemma \ref{lemma:u_b_derivative_bound} is uniform in $b$, we can repeat the proof of Lemma \ref{BGNTL7.4} with $\mu=\mu_n$ and $b=b_n$ to obtain
$$
\mu_n \dfrac{\partial u_{\mu_n,+}}{\partial b}(b_n;a,b_n)- 2 \dfrac{u_{\infty,+}^{\prime \prime}(b) - (u_{\infty,+}^\prime (b))^2}{u_{\infty,+}^\prime (b)} \rightarrow 0,
$$
which contradicts \eqref{eq:contradiction_BGNTL7.6}.
\end{proof}

For the decreasing solution, in the same fashion, we can prove the following.

\begin{lemma}\label{lemma:uniform_decreasing}
Fix $b\in (0,1]$ and $0<A_1<A_2< b$. We have
\[
\lim_{\mu\to\infty} \sup_{a\in [A_1,A_2]} \left[ \frac{e^{\mu(u_{\mu,-}(a;a,b)-1)}}{\mu}
-\frac{(u_{\infty,-}'(a;a,b))^2}{2}\right]=0,
\]
and
\[
\lim_{\mu\to\infty} \sup_{a\in [A_1,A_2]} \left[  \mu \dfrac{\partial u_{\mu,-}}{\partial a}(a;a,b) -
2 \dfrac{u_{\infty,-}^{\prime \prime}(a) - (u_{\infty,-}^\prime (a))^2}{u_{\infty,-}^\prime (a)}\right] =0.
\]
\end{lemma}

\begin{theorem}\label{thm:C1convergence}
Fix $0\leq a<b\leq1$. Let $L_{\mu}$ and $L_{\infty}$ be defined as in \eqref{defLmu} and \eqref{defLinf} respectively. For every $\varepsilon>0$ we have
$$
L_\mu (.;a,b)\rightarrow L_\infty (.; a,b)\ in\ C^1(a+\varepsilon ,b-\varepsilon).
$$
\end{theorem}
\begin{proof}
Given $\varepsilon>0$, we can take $B_1=a+\varepsilon$, $B_2=b$, $A_1=a$ and $A_2=b-\varepsilon$ in the previous Lemmas.
By Lemmas \ref{lemma:pohozaev_uniform} and \ref{lemma:uniform_decreasing} we have
\[
\lim_{\mu\to0} \sup_{s\in [a+\varepsilon,b-\varepsilon]} |L_\mu (s;a,b)- L_\infty (s; a,b)|=0.
\]
Concerning the derivative, notice that
\[
\dfrac{\partial}{\partial b}\left(\dfrac{e^{\mu(u_{\mu,+}(b;a,b)-1 )}}{\mu} \right)=
\frac{e^{\mu(u_{\mu,+}(b;a,b)-1)}}{\mu} \mu\frac{\partial u_{\mu,+}}{\partial b}(b;a,b).
\]
Using Lemmas \ref{lemma:pohozaev_uniform} and \ref{BGNTL7.6}, we have
\begin{multline}
\lim_{\mu\to\infty} \sup_{b\in [B_1,B_2]} \left|
\dfrac{\partial}{\partial b}\left(\dfrac{e^{\mu(u_{\mu,+}(b;a,b)-1 )}}{\mu} \right) \right. \\
\left. -\dfrac{(u_{\infty,+}^\prime (b;a,b) )^2}{2}  2 \dfrac{u_{\infty,+}^{\prime \prime}(b;a,b)  - (u_{\infty,+}^\prime (b;a,b) )^2}{u_{\infty,+}^\prime (b;a,b) } \right|=0
\end{multline}
On the other side, it is proved in \cite[Theorem 7.8]{BonheureGrossiNorisTerracini2015} that
$$
\dfrac{\partial}{\partial b}\left(\dfrac{(u_{\infty,+}^\prime (b;a,b))^2}{2} \right)=u_{\infty,+}^\prime (b;a,b) \left[u_{\infty,+}^{\prime\prime} (b;a,b) - (u_{\infty,+}^\prime (b;a,b))^2\right].
$$
Since an analogous result holds for $u_{\mu,-}(a;a,b)$, this concludes the proof.
\end{proof}

\begin{corollary}\label{coro:bar_s_continuous}
Given $0\leq\bar{a}<\bar{b}\leq1$, there exists $\varepsilon>0$ and $\bar{\mu}$ such that
\begin{equation}\label{eq:s_bar_mu_regular}
\bar{s}_\mu(a,b) \text{ is of class } C^1 \text{ for } (a,b)\in [\bar{a}-\varepsilon,\bar{a}+\varepsilon,\bar{b}-\varepsilon,\bar{b}+\varepsilon],
\end{equation}
\begin{equation}\label{eq:u_mu1layercontinuous1}
u_{\mu,1-layer}(b;a,b) \text{ is continuous for } (a,b)\in [\bar{a}-\varepsilon,\bar{a}+\varepsilon,\bar{b}-\varepsilon,\bar{b}+\varepsilon],
\end{equation}
\begin{equation}\label{eq:u_mu1layercontinuous2}
u_{\mu,1-layer}(a;a,b) \text{ is continuous for } (a,b)\in [\bar{a}-\varepsilon,\bar{a}+\varepsilon,\bar{b}-\varepsilon,\bar{b}+\varepsilon],
\end{equation}
for every $\mu\geq\bar{\mu}$.
\end{corollary}
\begin{proof}
In order to prove \eqref{eq:s_bar_mu_regular}, recall that $\bar{s}_{\mu}$ is implicitly defined by the relation $L_\mu (\bar{s}_{\mu}; a,b)=0$. It is proved in \cite[Theorem 6.1]{BonheureGrossiNorisTerracini2015} that 
\[
\frac{\partial}{\partial s} L_\infty(s;a,b)>0.
\]
The proof therein is for $N\geq 3$, but it can be repeated without changes in the case $N=2$ by making use of Lemma \ref{lemma:appendix1}.
Then Theorem \ref{thm:C1convergence} implies that for every $\varepsilon>0$ there exists $\bar{\mu}$ such that
\[
\frac{\partial}{\partial s} L_\mu(s;a,b)>0 \quad\text{ for } s\in [a-\varepsilon,b+\varepsilon], \quad \mu\geq\bar{\mu}.
\]
Since $\bar{s}_\mu\to \bar{s}_\infty$, which lies in the interior of the interval $(a,b)$, we deduce that
\[
\frac{\partial}{\partial s} L_\mu(\bar{s}_\mu;a,b)>0 
\]
for $\mu$ sufficiently large. The Implicit Function Theorem then implies that $\bar{s}_{\mu}(a,b)$ is locally of class $C^1$.

Next recall that, by definition,
\[
u_{\mu,1-layer}(b;a,b)= u_{\mu,-}(b;\bar{s}_\mu(a,b),b), \quad
u_{\mu,1-layer}(a;a,b)=u_{\mu,+}(a;a,\bar{s}_\mu(a,b)).
\]
Then \eqref{eq:u_mu1layercontinuous1} and \eqref{eq:u_mu1layercontinuous2} follow by combining Lemma \ref{lemma:C1_dependence} with \eqref{eq:s_bar_mu_regular}.
\end{proof}

\begin{corollary}\label{coro:uniform_converge_1_layer}
For every $\varepsilon>0$ we have that $u_{\mu,+}(a;a,\bar{s}_{\mu}(a,b)) \to u_{\infty,+}(a;a,\bar{s}_{\infty}(a,b))$ and  $u_{\mu,-}(b;\bar{s}_{\mu}(a,b),b) \to u_{\infty,-}(b;\bar{s}_{\infty}(a,b),b)$ uniformly in the set $\{\varepsilon<a<a+\varepsilon<b<1-\varepsilon\}$ for every $\mu\geq\bar{\mu}$.
\end{corollary}
\begin{proof}
Proceeding as in \cite[Lemma 5.9]{BonheureGrossiNorisTerracini2015}, one can see that to prove that $u_{\mu ,+}(a; a ,s_\mu (a,b))\rightarrow u_{\infty ,+}(a; a s_\infty (a,b))$ uniformly in $a$, one only needs to prove that $u_{\mu ,+}(.; a ,s_\mu (a,b))$ is equicontinuous in $a$, which is implied by
$$\left\|\dfrac{\partial}{\partial a}(u_{\mu ,+}(.; a ,s_\mu (a,b)) ) \right\|_\infty\leq C,$$
for some constant $C$ not depending on $\mu$ and $a$. We already know proceeding as in Lemma \ref{lemma:u_b_derivative_bound} that $\left\|\dfrac{\partial}{\partial a}(u_{\mu ,+}(.; a ,b) ) \right\|_\infty \leq C$. Thus, the result will follow if we can prove that $\left\|\dfrac{\partial}{\partial a} s_\mu (a,b)\right\|_\infty \leq C$. Let us recall that $s_\mu (a,b)$ is defined implicitly by the relation $L_\mu (s_\mu (a,b);a,b)=0$. Taking the derivative with respect to $a$, we find
$$\partial_s L(s_\mu ; a,b) \partial_{a}s_\mu +\partial_a L(s_\mu ;a,b)=0.$$
So
$$\partial_a s_{\mu} = - (\partial_s L(s_\mu ; a,b))^{-1} \partial_a L(s_\mu ;a,b). $$
We have $(\partial_s L(s_\mu ; a,b))^{-1}<\varepsilon^{-1}$ not depending on $\mu$ and $a$ since $\partial_s L_\infty (s_\infty ; a,b)>\tilde{\varepsilon}$ and $L_\mu (.;a,b )\rightarrow L_\infty (. ; a,b)$ uniformly in $C^1(s_\infty -\varepsilon_1 , s_\infty +\varepsilon_1 )$. We also have, proceeding as in Theorem \ref{thm:C1convergence},
$$\left\|\partial_a L(s_\mu ;a,b)\right\|\leq C.$$
The proof follows.
\end{proof}

\subsection{Proof of Theorem \ref{thm:variational}}\label{subsec:klayer}
Given $k\in \N_0$, let
$$
T=\left\{(\b_1,\ldots ,\b_{k-1})\in \R^{k-1}: 0=\b_0<\b_1<\ldots<\b_{k-1}<\b_k=1 \right\},
$$
In each interval $[\b_{j-1},\b_j]$, we consider the 1-layer solution $u_{\mu,1-layer}(r;\b_{j-1},\b_j)$ constructed in Theorem \ref{BGNT6.1} and we denote by $\a_{\mu,j}(\b_{j-1},\b_j)$ its unique maximum point. Recall that
\begin{equation}
\a_{\mu,j}\to \a_{\infty,j}= \a_{\infty,j}(\b_{j-1},\b_j), \quad\text{with}\quad
\left. \left( \frac{G(r,r;\b_{j-1},\b_j)}{r^{N-1}} \right)'\right|_{r=\a_{\infty,j}}=0.
\end{equation}

We first look for the solutions of point (i) of the statement, having $k$ interior maximum points. We aim to find (for $\mu$ sufficiently large) a zero of the function $M_\mu =(M_\mu^{(1)},\ldots , M_\mu^{(k-1)}):T\rightarrow \R^{k-1}$, defined as
$$
M_\mu^{(j)}(\b_1,\ldots , \b_{k-1})
= u_{\mu,1-layer}(\b_j;\b_j,\b_{j+1})-u_{\mu,1-layer}(\b_j;\b_{j-1},\b_j),
$$
for $j=1,\ldots,k-1$. We also define $M_\infty=(M_\infty^{(1)},\ldots , M_\infty^{(k-1)}):T\rightarrow \R^{k-1}$ as
$$M_\infty^{(j)}(\b_1,\ldots , \b_{k-1})
= \dfrac{G(\b_j,\a_{\infty, j+1};\b_j,\b_{j+1})}{G(\a_{\infty, j+1 },\a_{\infty, j+1};\b_j,\b_{j+1})}-\dfrac{G(\b_j,\a_{\infty ,j};\b_{j-1},\b_j)}{G(\a_{\infty ,j},\a_{\infty, j};\b_{j-1},\b_{j})},
$$
for $j=1,\ldots,k-1$.
It is proved in \cite[Theorem 2.14]{BonheureGrossiNorisTerracini2015} that, given $P\in T$, $M_\infty$ is homotopic to $Id-P$ in $T$,. By the excision property of the topological degree, there exists an open set $U$, with $\overline{U}\subset T$, such that
$$
deg (M_\infty,U,0)= 1.
$$
Since $M_\mu \rightarrow M_\infty$ uniformly in $U$ (Corollary \ref{coro:uniform_converge_1_layer}) and $M_\mu$ is continuous in $U$ for $\mu$ sufficiently large (Corollary \ref{coro:bar_s_continuous}), the Rouch\'e's property of the topological degree (see for example \cite[Corollary 3.4.2]{DincaMawhin2009}) yields to
$$
deg (M_\mu,U,0)= 1.
$$
Therefore $M_\mu$ has an interior zero $(\bar{\b},\ldots,\bar{\b}_{k-1}$ and
\[
u= u_{\mu,1-layer}(r;\bar{\b}_{j-1},\bar{\b}_j), \quad r\in (\bar{\b}_{j-1},\bar{\b}_j), \quad j=1,\ldots,k
\]
is the required solution.

In order to find the solutions of point (ii) of the statement, having $k-1$ interior maximum points and one maximum point at $r=1$, we similarly define $\tilde{M}_\mu =(\tilde{M}_\mu^{(1)},\ldots , \tilde{M}_\mu^{(k-1)}):T\rightarrow \R^{k-1}$ as $\tilde{M}_\mu^{(j)}=M_{\mu}^{(j)}$, for $j=1,\ldots , k-2$ and $\tilde{M}_{\mu}^{(k-1)}= u_{\mu,+}(\b_{k-1};\b_{k-1},1)-u_{\mu,1-layer}(b_{k-1};b_{k-2},b_{k-1}) $. It can be proved as in \cite[Theorem 2.14]{BonheureGrossiNorisTerracini2015} that $\tilde{M}_\infty$ is homotopic to $Id-P$, hence we can proceed as above.

\appendix

\section{Green's function in dimension 2}\label{sec:appendix}

In this appendix we prove the analogous of \cite[Proposition 2.1]{BonheureGrossiNorisTerracini2015} in the 2-dimensional case (the case $N=2$ is not treated in \cite{BonheureGrossiNorisTerracini2015}).

\begin{lemma}\label{lemma:appendix1}
There exist two positive, linearly independent solutions $\zeta \in C^2 ((0,1])$ and $\xi \in C^2 ([0,1])$ of the equation
$$-u^{\prime \prime}-\dfrac{1}{r}u^\prime +u=0\ in\ (0,1),$$
satisfying
\begin{equation}\label{eq:xi_zeta_N2}
\xi^\prime (0)=\zeta^\prime (1)=0,\quad r(\xi^\prime (r) \zeta (r) - \xi (r) \zeta^\prime (r))=1, \ \forall r\in (0,1].
\end{equation}
Moreover, $\xi$ is bounded and increasing in $[0,1]$, $\zeta$ is decreasing in $(0,1]$ and 
$$\xi (0)=1,\quad \lim_{r\rightarrow 0^+}\dfrac{\zeta (r)}{-\ln r}=1,\quad \lim_{r\rightarrow 0^+ }(-r\zeta^\prime (r))=1.$$
As a consequence, the Green function defined in \eqref{eq:L_mathcal} (for $N=2$) can be written as follows
\begin{equation}\label{eq:greenN=2}
G(r,s)=\left\{\begin{array}{ll}
s^{N-1}\xi(r)\zeta(s)\quad\text{for }r\leq s \\
s^{N-1}\xi(s)\zeta(r)\quad\text{for }r> s.
\end{array}\right.
\end{equation}
\end{lemma}
\begin{proof}
Let $\xi (r):= I_0 (r)$ be the modified Bessel function of the first kind (see ($10.25$) of \cite{bessel}). It is well-known that $\xi$ is positive, bounded, increasing, and that $\xi^\prime (0)=0,\ \xi(0)=1$. Let $s=\ln r$, $s\in (-\infty ,0]$ and $\varphi (s)= \xi (e^s)$. In this new variable, we have that
$$
\begin{cases}
\varphi^{\prime \prime}+\varphi e^{2s}=0 \quad \text{ in } (-\infty ,0)\\
\displaystyle\lim_{s\rightarrow -\infty} \varphi (s)=1, \quad
\displaystyle\lim_{s\rightarrow -\infty} \varphi^\prime (s)=0,\\
\varphi (0)=\xi (1)>0, \quad \varphi^\prime (0)=\xi^\prime (1)>0.
\end{cases}
$$
We set 
$$\psi (s) := \varphi (s) \left\{\dfrac{1}{\varphi (0) \varphi^\prime (0)}+ \int_s^0 \dfrac{1}{\varphi^2 (t)}dt \right\}.$$
By direct calculations one can check that $\psi^\prime (0)=0$, 
\begin{equation*}
-\psi^{\prime \prime}+e^{2s}\psi =0,  \text{ and } 
\varphi^\prime (s) \psi (s) -\psi^\prime (s) \varphi (s) =1,
\end{equation*}
for every $s\in (-\infty ,0)$.
We also have $\psi^\prime (s)= - \int_s^0 \psi(t) e^{2t} dt <0$. Moreover, the relation $\displaystyle\lim_{s\rightarrow -\infty} \varphi (s)=1$ implies 
$$
\lim_{s\rightarrow -\infty} \psi^\prime (s)=-1,
$$ 
and using L'hospital rule,
$$
\lim_{s\rightarrow -\infty} \dfrac{\psi (s)}{s}=\lim_{s\rightarrow -\infty}\psi^\prime (s)=-1.
$$
Thus letting $\zeta (r):= \psi (\ln r)$, we have all the claimed properties.

Finally, in order to prove \eqref{eq:greenN=2}, we have to show that, for all $\varphi \in C^\infty([0,1])$,
\begin{equation}\label{eq:Green_delta}
\int_0^1 \left(\frac{\partial G}{\partial r}(r,s) \varphi'(r)+ G(r,s)\varphi(r) \right)r^{N-1} \,dr =s^{N-1} \varphi(s).
\end{equation}
By the defintion $G$ in \eqref{eq:greenN=2}, the left hand side of \eqref{eq:Green_delta} rewrites as
\begin{multline*}
\int_0^s \left( \xi'(r)\varphi'(r) +  \xi(r) \varphi(r) \right) \zeta(s) s^{N-1} r^{N-1} \,dr \\
+ \int_s^1 \left( \zeta'(r)\varphi'(r) + \zeta(r) \varphi(r) \right) \xi(s) s^{N-1} r^{N-1} \,dr.
\end{multline*}
We integrate by parts and we use the equation satisfied by $\xi$ and $\zeta$ and the respective boundary conditions, to obtain
\begin{multline*}
\int_0^s s^{N-1} \zeta(s) \varphi(r) \left[-(r^{N-1}\xi'(r))'+\xi(r)r^{N-1} \right] \,dr + s^{2N-2}\xi'(s)\zeta(s)\varphi(s) \\
+\int_s^1 s^{N-1} \xi(s) \varphi(r) \left[-(r^{N-1}\zeta'(r))'+\zeta(r)r^{N-1} \right] \,dr - s^{2N-2}\xi(s)\zeta'(s)\varphi(s) \\
=s^{2N-2}\xi'(s)\zeta(s)\varphi(s) - s^{2N-2}\xi(s)\zeta'(s)\varphi(s).
\end{multline*}
Then using \eqref{eq:xi_zeta_N2} we obtain $s^{N-1} \varphi(s)$, so that \eqref{eq:Green_delta} is proved.
\end{proof}


\end{document}